\documentclass[11pt,a4paper]{article}
\usepackage[utf8]{inputenc}
\usepackage[top=0.99in, bottom=0.99in, left=0.98in, right=0.98in]{geometry}
\usepackage{amsmath,amsthm,amssymb,amscd,dsfont,bm,epic,accents}
\usepackage{graphicx,subfig}
\usepackage[all]{xy}
\usepackage{tikz}
\usetikzlibrary{intersections}

\usepackage[backend=bibtex,style=numeric,sorting=nyt,isbn=false,doi=false,url=false,maxbibnames=99]{biblatex}
\DeclareFieldFormat{pages}{#1}
\renewbibmacro{in:}{}

\usepackage[titletoc,title]{appendix}
\usepackage{comment}
\usepackage{tkz-euclide}
\usetkzobj{all}
\usepackage[titletoc,title]{appendix}

\bibliography{bib}
\theoremstyle{plain}
\newtheorem{thm}{Theorem}[section]
\newtheorem{prop}[thm]{Proposition}
\newtheorem{cor}[thm]{Corollary}
\newtheorem{lem}[thm]{Lemma}

\theoremstyle{definition}
\newtheorem{dfn}[thm]{Definition}

\newtheorem{exa}[thm]{Example}

\theoremstyle{remark}
\newtheorem{rmk}[thm]{Remark}

\DeclareMathOperator{\Ima}{Im}
\DeclareMathOperator{\Id}{Id}
\DeclareMathOperator{\Gr}{Gr}
\DeclareMathOperator{\Res}{Res}
\DeclareMathOperator{\ev}{ev}

\def\RW{\operatorname{FJRW}}
\def\GW{\operatorname{GW}}
\def\Orl{\operatorname{Orl}}
\def\ch{\operatorname{ch}}

\def\Db{\text{D}^\text{b}}
\def\DMF{\text{DMF}}
\def\CR{{\mathbb C^*_R}}
\def\C{\mathbb C}

\def\P{\mathbb P}
\def\la{\lambda}
\def\Ocal{\mathcal O}

\begin{document}
	\title{\textbf{Landau--Ginzburg/Calabi--Yau correspondence for a complete intersection via matrix factorizations}}
	\author{Yizhen Zhao}
	\date{}
	\maketitle
	\begin{abstract}
	By generalizing the Landau--Ginzburg/Calabi--Yau correspondence for hypersurfaces, we can relate a Calabi--Yau complete intersection to a hybrid Landau–Ginzburg model: a family of isolated singularities fibered over a projective line. In recent years Fan, Jarvis, and Ruan have defined quantum invariants for singularities of this type, and Clader and Clader--Ross have provided a equivalence between these invariants and Gromov--Witten invariants of complete intersections. For Calabi-Yau complete intersections of two cubics, we show that this equivalence is directly related --- via Chen character --- to the equivalences between the derived category of coherent sheaves and that of matrix factorizations of the singularities. This generalizes Chiodo--Iritani--Ruan's theorem matching Orlov's equivalences and quantum LG/CY correspondence for hypersurfaces. 
	\end{abstract}
	\tableofcontents
\section{Introduction}

The Landau--Ginzburg/Calabi--Yau (LG/CY) correspondence in string theory describes a relationship between the sigma model on a Calabi--Yau hypersurface and the Landau--Ginzburg model whose potential is the defining equation of the Calabi--Yau variety. Following Witten \cite{witten1993phases}, we can present 
the LG/CY correspondence ia a purely algebro-geometric way starting from a variation of stability conditions in geometric invariant theory (GIT). From this point of view, we can naturally generalize the LG/CY correspondence to the Calabi--Yau complete intersections. 

The variation of stability conditions leads to two different curve-counting theories. Analytic continuation naturally allows us to compare them. A natural question arises: what is the interpretation of the linear transformation matching the generating functions encoding the two theories? The answer given in this paper is an equivalence of triangulated categories known as Orlov equivalence applied to the two GIT quotients. More precisely GIT quotients are classically interpreted as chambers and the transition between them is usually phrased in terms of window-transitions. Each of these transitions is related to a specialization of Orlov's functor (see \S \ref{secorlovdef}). This is a mathematical object of independent interest not directly related to curve-counting theories. In particular it sheds new light on the LG/CY correspondence. 


\subsection{The problem}
We start from $r$ homogeneous polynomials $W_1,\dots, W_r$ 
of the same degree $d$ defining a smooth complete intersection in 
$\P^{N-1}$. The complete intersection is Calabi--Yau\footnote{Here the Calabi--Yau condition is meant in the weak sense: the canonical bundle $\omega$ is trivial.} as soon as $dr$ equals $N$. Following a standard procedure (see Witten \cite{witten1993phases}, we also refer to Herbst--Hori--Page \cite{herbst2008phases})
we can cast this setup within a $\C^*$-action as follows. 
Consider a $\mathbb C^*$-action on the vector space $V=\mathbb C^N\times \mathbb C^r=\operatorname{Spec}\mathbb C[x_1,\dots,x_N,p_1,\dots,p_r]$ with weight $1$ on the first $N$ variables, and weight $-d$ on the 
following $r$ variables 
$$\lambda \cdot (x_1,\dots,x_N,p_1,\dots,p_r)= (\la x_1,\dots,\la x_N,\la ^{-d}p_1,\dots,\la^{-d}p_r).$$
GIT provides a systematic description of the geometric quotients that can be obtained 
from the action $\C^*\times V\to V$. Indeed we 
can choose two different GIT stability conditions to identify 
two maximal open sets within $\Omega\subset V$ whose quotient $[\Omega/\C^*]$ 
is a smooth Deligne--Mumford stack. Indeed, for 
$$\Omega_+=V\setminus(x_1=\dots=x_N=0) \qquad \text{and} \qquad \Omega_-=V\setminus(p_1=\dots=p_r=0),$$
we obtain the total space $X_+$ of the vector bundle $\Ocal(-d)^{\oplus r}$ on $\P^{N-1}$ 
and the total space $X_-$ of the vector bundle $\Ocal(-1)^{\oplus N}$ on $\P(d,\dots,d)$ (the weighted projective stack 
with an overall stabilizer $\mu_d$, whose coarse space equals $\P^{r-1}$). 

The Calabi--Yau complete intersection, or rather its cohomology, arises 
as the  cohomology of $X_+$ relative to the generic fiber of 
\begin{equation}\label{diagramcy}
\textstyle{\sum_{j=1}^r p_jW_j\colon X_+\to \C.}
\end{equation}
For $X_-$, the same procedure yields a family of isolated singularities over $\P(d,\dots,d)$:
\begin{equation}\label{diagramlg}
\xymatrix{
	X_- \ar[rr]^{\textstyle{\sum_{j=1}^r p_jW_j}} \ar[d] &&\mathbb C\\
	\P(d,\dots,d)
}.
\end{equation}
We call it a Landau--Ginzburg model $(X_-,W_1,\dots,W_r)$. It is the analogue of the $\mu_d$-invariant polynomial
$$
\xymatrix{
	[\mathbb C^N/ \mu_d] \ar[r]^{W}\ar[d] &\mathbb C\\
	\text{\quad\qquad}B\mu_d=\mathbb P(d)
}
$$
defining the Landau--Ginzburg singularity model of Calabi--Yau hypersurface in \cite{chiodo2014landau,chiodo2010landau}.

The relative cohomology groups $H_{\GW}$ of (\ref{diagramcy}) and $H_{\RW}$ of (\ref{diagramlg}) are isomorphic (see Chiodo--Nagel \cite{chiodo2015hybrid} in higher generality). 
The subject of this paper is an enhanced correspondence in terms of curve-counting theories.

\subsection{Curve-counting theories}

For the Calabi–Yau complete intersection, we consider the Gromov–Witten (GW) theory. For the Landau–Ginzburg model $(X_- , W_1, \dots, W_r)$, a curve-counting theory was constructed by Fan–Jarvis–Ruan in \cite{fan2007witten,fan2008geometry,fan2013witten,fan2015mathematical}. We will use their definition, which we will refer to as FJRW theory, see \cite{fan2015mathematical} and \S \ref{secgt}. Since the first definition of FJRW theory \cite{fan2007witten}, several alternative constructions have been provided: Polishchuk--Vaintrob \cite{polishchuk2011matrix}, Chang--Li--Li \cite{chang2015witten}, Ciocan-Fontanine--Favero--Gu{\'e}r{\'e}--Kim--Shoemaker \cite{ciocan2018fundamental}.

According to the LG/CY correspondence, it is natural to conjecture that the genus-$0$ GW theory of the Calabi--Yau complete intersection and the genus-$0$ FJRW theory of $(X_-,W_1,\dots,W_r)$ are equivalent in the following sense. The genus-$0$ GW theory of the Calabi--Yau complete intersection is determined by a function $I_{\GW}$ taking values in the vector space $H_{\GW}$, and the genus-$0$ FJRW theory of $(X_-,W_1,\dots,W_r)$ is determined by a function $I_{\RW}$ taking values in the vector space $H_{\RW}$. The two theories are equivalent in the sense that $I_{\GW}$ matches $I_{\RW}$ up to an analytic continuation and a linear map. The above conjecture was proven in the hypersurface case (\textit{i.e.} $r=1$) by Chiodo--Ruan \cite{chiodo2010landau}, Chiodo--Iritani--Ruan \cite{chiodo2014landau} and Lee--Priddis--Shoemaker \cite{lee2016proof}, and generalized to certain complete intersection (\textit{i.e.} $r>1$) by Clader \cite{clader2017landau} and Clader--Ross \cite{clader2017sigma}. Chiodo--Iritani--Ruan \cite{chiodo2014landau} provided a geometric interpretation in the hypersurface case in terms of Orlov functor. We focus on the complete intersection case.

 In this paper, we simplify the $I$-functions with the help of $\Gamma$-classes introduced by Iritani \cite{iritani2009integral}, and compute the explicit form of a family of linear maps 
 $$\mathbb U_l\colon H_{\RW} \to H_{\GW}$$ 
 indexed by $l\in \mathbb Z$ relating $I_{\RW}$ and $I_{\GW}$. Then, we can relate $\mathbb U_l$ to equivalences of categories.

\subsection{Relation to equivalences of categories}
Orlov \cite{orlov2009derived} proved that there is a family of equivalences of triangulated categories indexed by $\mathbb Z$ between the derived category $\DMF$ of graded matrix factorizations \cite{orlov2009derived,segal2011equivalences,shipman2012geometric} and the derived category $\Db$ of coherent sheaves on a Calabi--Yau hypersurface. We generalize Orlov's result to the complete intersections. We construct such equivalences $$\Orl_t\colon \DMF \to \Db$$ by compositing two functors. One of them is due to Segal \cite{segal2011equivalences}; the other one is due to Shipman \cite{shipman2012geometric}, see also Isik \cite{isik2013equivalence} for an alternative construction.

A physics paper \cite{herbst2008phases} by Herbst--Hori--Page predicts that the LG/CY correspondence is related to equivalences of categories. In the hypersurface case, it was verified by Chiodo--Iritani--Ruan in \cite{chiodo2014landau}. We study the case of complete intersection of two cubics in $\mathbb P^5$.

\begin{thm}[Theorem \ref{maintheorem}]
	In the case $N=6,d=3,r=2$, we can find a subcategory $\mathcal G$ of $\DMF$, such that the following diagram commutes:
	\begin{equation*}
	\xymatrix{
		&\mathcal G \ar[rrr]^{\otimes \mathcal O(-3)\circ\Orl_{t-3}}\ar[d]_{\ch}& &&\Db\ar[d]^{\ch}\\
		&H_{\RW} \ar[rrr]^{\mathbb U_t}&& &H_{\GW}.
	}
	\end{equation*}
	Here the two vertical arrows represent the Chern character in the corresponding categories.
\end{thm}
Our method provides an algorithm for any $N,d,r$. Howerer, the complexity of the analytic continuation prevent us from giving a general proof. So we restrict ourselves to the case $N=6,d=3,r=2$,  the simplest case where the complete intersection is a Calabi--Yau 3-fold but not a hypersurface.

\subsection{Organization of the paper}
We introduce GW theory and FJRW theory in \S \ref{sectwotheories}. In \S \ref{secanacon}, we carry out the analytic continuation and get linear maps relating the two theories. We introduce the category of matrix factorizations and construct Orlov functor in \S \ref{secorl}. The main result is proven in \S \ref{seccom}.
\subsection{Acknowledgements}
We would like to thank E. Clader, J. Gu{\'e}r{\'e}, D. Ross and M. Shoemaker for helpful conversations and suggestions.
A special thank to A. Polishchuk for pointing out to us references to Isik \cite{isik2013equivalence}, Segal \cite{segal2011equivalences} and Shipman \cite{shipman2012geometric}. A. Chiodo also participated
crucially in the development of the paper, by way of countless
conversations, support, and advice.
\section{Terminology}

We denote by $\P(w_1,\dots,w_k)$ the weighted projective stack with weights $w_1,\dots,w_k$. It is quotient stack $[(\C^k- \{0\})/\C^*]$, where the $\C^*$-action on $\C^k -\{0\}$ is given by
$$\lambda \cdot (x_1,\dots,x_k)= (\la^{w_1} x_1,\dots,\la^{w_k} x_k).$$
Consider a reductive group $G$, with an action on a scheme $X$, and a chosen linearization $\theta$. We denote by $X_{G.\theta}^{\text{ss}}$ the corresponding semi-stable point set. The corresponding GIT quotient is denoted by $[X/\!\!/\!_{\theta} G]$.






\section{Two parallel theories}\label{sectwotheories}
In this section, we introduce the two parallel theories coming from a variation of stability conditions in geometric invariant theory (GIT). One of them is the genus-$0$ Gromov--Witten (GW) theory of a Calabi--Yau complete intersection, the other one is the genus-$0$ Fan--Jarvis--Ruan--Witten (FJRW) theory for the Landau--Ginzburg model.

\subsection{Input data}\label{input}
Let $W_1,\dots,W_r$ be a collection of degree-$d$ quasihomogeneous polynomials in the variables $x_1,\dots,x_N$, where $x_i$ has weight $w_i$. The weights $w_1,\dots,w_N$ are coprime. We require that the forms $dW_1,\dots,dW_r$ are linearly independent at the common $0$-locus of the polynomials $W_i$, except at the point $x_1=\dots=x_N=0$. Then
$$
W_1=\dots=W_r
$$
defines a complete intersection $X_{d,\dots,d}$ in the weighted projective stack $\mathbb P(w_1,\dots,w_N)$. The weights $w_1,\dots,w_N$ satisfy the Calabi--Yau condition
$$
\sum_{i=1}^N w_i=rd.
$$
By the adjunction formula, $X_{d,\dots,d}$ is Calabi--Yau in the sense that its canonical sheaf is trivial. We further require the Gorenstein condition to be satisfied:
$$
w_i\vert d,\quad 1\le i \le N.
$$
\begin{rmk}
	We recall that the Gorenstein condition is needed for any computation of Gromov--Witten invariants. We refer the reader to \cite{coates2012lefschetzcanfail,guere2016withoutconavity} illustrating that the Lefschetz principle may fail otherwise.
\end{rmk}
Following a standard procedure (see \cite{herbst2008phases,witten1993phases})
we can recast this setup as follows. 
Let $G=\mathbb C^*$, consider a $G$-action on the vector space 
$$V=\mathbb C^N\times \mathbb C^r=\operatorname{Spec}\mathbb C[x_1,\dots,x_N,p_1,\dots,p_r]$$ with weights $w_i$ on the first $N$ variables $x_i$, and weight $-d$ on the 
following $r$ variables 
$$\lambda \cdot (x_1,\dots,x_N,p_1,\dots,p_r)= (\la^{w_1} x_1,\dots,\la^{w_N} x_N,\la ^{-d}p_1,\dots,\la^{-d}p_r).$$
Since  $w_1,\dots,w_N$ are coprime, we can regard $G$ as a subgroup of $\operatorname{GL} (V)$.

There is another $\C^*$-action. We denote this $\C^*$ by $\CR$. The group $\CR$ acts on  
$V$ with weight $0$ on the first $N$ variables, and weight $1$ on the 
following $r$ variables: 
$$\mu \cdot (x_1,\dots,x_N,p_1,\dots,p_r)= ( x_1,\dots, x_N,\mu p_1,\dots,\mu p_r).$$
We can also regard $\CR$ as a subgroup of $\operatorname{GL} (V)$. Let $\Gamma$ be the subgroup of $\operatorname{GL}(V)$ generated by $G$ and $\CR$. Then we have an isomorphism
$$
\Gamma=G\CR \cong G \times \CR.
$$
Denote by $
\xi\colon \Gamma \to G
$ and $
\zeta \colon \Gamma \to \CR
$ the first and second projections.

Set
$$
W=p_1W_1+\dots+p_rW_r,
$$
then $W$ is a function over $V$ invariant under the $G$-action.

\begin{rmk}\label{goodlift}
	It is more natural to start from quasihomogeneous polynomials $W_1,\dots,W_r$ with different degrees $d_1,\dots,d_r$. However, we do not have well-defined enumerative theory (see definition \ref{fjrwt}) in this case due to the lack of a ``good lift''.
	We say a $\Gamma$-character $\hat{\theta}$ is a good lift of a $G$-character $\theta$ if it is compatible with the inclusion $G\le \Gamma$, and satisfies $$V_{\Gamma,\hat{\theta}}^{\text{ss}}=V_{G,\theta}^{\text{ss}} .$$
\end{rmk}

\subsection{Two different GIT quotients}\label{setup}
We consider the GIT quotient of $V$ with respect to the $G$-action. Each character of $G$ defines a linearlization of the trivial line bundle over $V$. There are two types of $G$-characters.
\begin{itemize}
	\item 
We can take a positive $G$-character, \textit{i.e.}
$$
\la \mapsto \la^k, \quad k>0.
$$
We denote the corresponding linearlization by $\theta_+$. Then the semi-stable point set is $$V_{G,\theta_+}^{\text{ss}}=(\C^N-\{0\})\times \C^r.$$ We denote the corresponding GIT quotient $[V/\!\!/\!_{\theta_+} \mathbb G]$ by $X_+$. Then, the quotient stack
$$
X_+=[(\C^N-\{0\})\times \C^r/G]
$$
is the total space of the vector bundle $\mathcal O_{\mathbb P(w_1,\dots,w_N)}(-d)^{\oplus r}$.
\item
We can take a negative $G$-character, \textit{i.e.}
$$
\la \mapsto \la^k, \quad k<0.
$$ 
We denote the corresponding linearlization by $\theta_-$. Then the semi-stable point set is $$V_{G,\theta_-}^{\text{ss}}=\C^N\times (\C^r-\{0\}).$$
 We denote the corresponding GIT quotient  $[V/\!\!/\!_{\theta_-} \mathbb G]$ by $X_-$. Then, the quotient stack
$$
X_-=[\C^N\times (\C^r-\{0\})/G]$$
is the total space of the vector bundle $\bigoplus_{i=1}^N\mathcal O_{\mathbb P(d,\dots,d)}(-w_i)$.
\end{itemize}

Note that $W$ is a function over $V$ which is invariant under the $G$-action. Hence, $W$ is a well-defined function on both $X_+$ and $X_-$. Let $F_\pm$ denote the Milnor fibers $W^{-1}(A)\subset X_\pm$ for a sufficiently large real number $A$.

\subsection{Hybrid theory} \label{secgt}
An enumerative theory is constructed in \cite{fan2015mathematical} for the input data as in \S \ref{input} and a choice of character of $G$. The character can be chosen to be positive or negative as in \S \ref{setup}.

An enumerative theory consists of the data of a state space, moduli spaces, and correlators. The state space is a graded vector space. The correlators are intersection numbers in the moduli spaces; they depend on the insertions coming from the state space.
\subsubsection*{State space}
The state space of the theory is defined to be the relative Chen--Ruan cohomology group (see Appendix) $$H_{\text{CR}}^*(X_\pm,F_\pm,\mathbb C)$$ with an addition shift $-2r$ in grading. So the component with grading $k$ of the state space is $H_{\text{CR}}^{k+2r}(X_\pm,F_\pm,\mathbb C)$. 
\begin{rmk}
	 When we choose a positive character, we have the isomorphisms
	$$
	H_{\text{CR}}^{*+2r}(X_+,F_+,\mathbb C)\cong H^{*+2r}(\mathbb P^{N-1},\mathbb P^{N-1} \backslash X_{d,\dots,d},\C)\cong H^*(X_{d,\dots,d},\C),
	$$
	where the first one comes from contraction, the second one is the Thom isomorphism. Note that $H^*(X_{d,\dots,d},\C)$ is the state space of GW theory for $ X_{d,\dots,d}$.
\end{rmk}
 
 In \cite{fan2015mathematical}, the theory is defined on a subspace of the state space. This subspace consists of classes of so-called compact type. In our situation, both $X_-$ and $X_+$ are total space of vector bundles. We denote the corresponding zero sections by $X_-^{\text{cp}}$ and $X_+^{\text{cp}}$. Following \cite{clader2017higherwallcrossing}, the subspaces of compact type are the image of the morphisms
 $$
 H^{*-2r}_{\text{CR}}(X_\pm,X_\pm\backslash X_\pm^{\text{cp}})\to  H^{*-2r}_{\text{CR}}(X_\pm,F_\pm).
 $$

\subsubsection*{Moduli space}
The moduli space in the theory is the moduli space of the following objects:
\begin{dfn}[\cite{fan2015mathematical}]\label{fjrwt}
An $\infty$-stable, $k$-pointed, genus-$g$ LG-quasimaps to the critical locus of $W$ consists the following data:
\begin{itemize}
	\item A prestable, $k$-pointed orbicurve $(\mathcal C,y_1,\dots, y_k)$ of genus $g$.
	\item A principal orbifold $\Gamma$-bundle $\mathcal P\colon \mathcal C \to B\Gamma$ over $\mathcal C$.
	\item A global section $\sigma \colon \mathcal C \to \mathcal P\times_{\Gamma}V$.
	\item An isomorphism $\kappa \colon \zeta_*\mathcal P\to \mathring{\omega}_{\log,\mathcal C}$ of principal $\mathbb C^*$ bundles, where $\mathring{\omega}_{\log,\mathcal C}$ is the principal bundle associated to the line bundle $\omega_{\log,\mathcal C}$.
\end{itemize} 
such that the following conditions are satisfied:
\begin{enumerate}
	\item The morphism of stack  $\mathcal P\colon \mathcal C \to B\Gamma$ is representable.
	\item\label{cond2} The image of the induced map $[\sigma]\colon \mathcal P \to V$ lies in the semistable locus (with respect to the $G$-action and chosen character) of the critical locus of $W$.
	\item The line bundle $\omega_{\log,\mathcal C}\otimes\sigma^*(\mathcal N)^\epsilon$ is ample for all sufficiently large $\epsilon$, where $\mathcal N$ is the line bundles over $[V/\Gamma]$ determined by a good lift (see remark \ref{goodlift}) of the chosen $G$-character.
\end{enumerate}
\end{dfn}
\begin{rmk}
Since $\Gamma \cong G \times \CR$, and $\zeta$ is just the second projection, giving a principal $\Gamma$ bundle $\mathcal P$ is the same as giving a line bundle $\mathcal L$ such that $\mathcal P \cong \mathring{\mathcal L}\times \mathring{\omega}_{\log,\mathcal C}$. Then we can write
$$\mathcal P\times_{\Gamma}V\cong \bigoplus_{i=1}^N \mathcal L^{\otimes w_i}\oplus \left(\omega_{\log,\mathcal C}\otimes L^{\otimes -d}\right)^{\oplus r}.$$
Thus giving a section of $\mathcal P\times_{\Gamma}V$ is the same as giving sections $s_i\in \Gamma(\mathcal C ,\mathcal L^{\otimes w_i})$ and $t_j\in \Gamma(\mathcal C ,\omega_{\log,\mathcal C}\otimes L^{\otimes -d})$ for $1\le i \le N$ and $1\le j \le r$.
\end{rmk}
In order to determine the semistable locus of the critical locus of $W$, we write
$$dW=\sum_{i=1}^r(p_idW_i+W_idp_i).$$
According to the nondegeneracy condition, the critical locus of $W$ is $$\{x_1=\dots=x_N=0\}\cup\{p_1=\dots=p_r=W_1=\dots=W_r=0\}.$$

\begin{itemize}
	\item When we choose a positive character, condition \ref{cond2} implies $$t_1=\dots=t_r=0$$ 
	and
	$$
	W_1(s_1,\dots,s_N)=\dots=W_r(s_1,\dots,s_N)=0.
	$$ 
	In this case the above data is equivalent to a stable map to the complete intersection $X_{d,\dots,d}$. Note that the theory is only defined on the classes of compact type. It coincides with the classic Gromov--Witten theory of $X_{d,\dots,d}$ restricted to the hyperplane section classes.
	\item When we choose a negative character, condition \ref{cond2} implies $$x_1=\dots=x_N=0.$$ 
	In this case the above data is equivalent to a map $f\colon \mathcal C \to \mathbb P^{r-1}$ together with an isomorphism $\phi \colon \mathcal L^{\otimes d} \cong \omega_{\log,\mathcal C}\otimes f^*\mathcal O(-1)$, and the theory constructed here coincides with Fan--Jarvis--Ruan--Witten theory for $(X_-,W_1,\dots,W_r)$, see \cite{clader2017landau}.
\end{itemize}

In this paper, we focus on the situation where $N=6,r=2,w_1=\dots=w_6=1,d=3,g=0$. It is the simplest  case where the complete intersection $X_{3,3}$ is a Calabi-Yau 3-fold but not a hypersurface. We recall the GW theory of $X_{3,3}$ and the FJRW theory for $(X_-,W_1,W_2)$ in the following sections. We focus on the the GW theory because the FJRW theory is totally parallel.

\subsection{Gromov--Witten theory of $X_{3,3}$} \label{secgwtheory}

In this subsection, We recall the full Gromov--Witten theory of $X_{3,3}$ first, then we restrict ourselves to the part coming from the ambient space $\mathbb P^5$ according to the Lefschetz principle. 
\subsubsection*{Full theory}
Let $\overline{\mathcal M}_{0,n}(X_{3,3},d)$ denote the moduli spaces of genus-$0$ degree-d n-marked stable maps to $X_{3,3}$. For each $s=1,\dots,n$, let 
$$
\ev_s \colon \overline{\mathcal M}_{0,n}(X_{3,3},d)\to X_{3,3}
$$
be the evaluation map at the s-th marked point, and 
$$
\psi_s\in H^2(\overline{\mathcal M}_{0,n}(X_{3,3},d))
$$
be the first Chern class of the universal cotangent line bundle at the s-th marked point. 

The state space of the entire Gromov--Witten theory of $X_{3,3}$ is $H^*(X_{3,3})$. For any choice of $\phi_1,\dots,\phi_n\in H^*(X_{3,3})$, $a_1,\dots,a_n \in \mathbb Z^{\ge 0}$ and $d \in \mathbb Z$, the corresponding Gromov--Witten invariant is defined as
$$\langle \phi_1\psi_1^{a_1},\dots,\phi_{n-1}\psi_{n-1}^{a_{n-1}},\phi_n\psi_n^{a_n}
\rangle_{0,n,d}^{\GW}:=\int_{[\overline{\mathcal M}_{0,n}(X_{3,3},d)]^{\text{vir}}}\prod_{s=1}^n(\psi_s^{a_s}\ev_s^*\phi_s).$$
We can define a generating function 
$$
\mathcal F^0_{\GW}(\boldsymbol{t}):=\sum_{n,d}\frac{Q^d}{n!}\langle \boldsymbol{t}(\psi_1),\dots,\boldsymbol{t}(\psi_{n-1}),\boldsymbol{t}(\psi_n)
\rangle_{0,n,d}^{\GW},
$$
where $\boldsymbol{t}=t_0+t_1z+t_2z^2+\dots\in H^*(X_{3,3})\otimes\mathbb C[z]$.
\begin{rmk}\label{Qe1}
	Let $p\in H^2(X_{3,3})$ be the hyperplane class. Denote the degree-2 part of $t_0$ by $t_0^2p$. We can take $t_0^2p$ out of the bracket repeatedly by divsor equation. Then $Q$ and $t_0^2p$ always appear together in the form $Qe^{t_0^2}$. So from next subsection, we set $Q=1$, and denote $e^{t_0^2}$ by $v$.
\end{rmk}

\subsubsection*{Givental's formalism}
Introduce the supervector space 
$$
\mathcal H^{\GW}=H^*(X_{3,3})\otimes \mathbb C((z^{-1}))
$$
of cohomology-valued Laurent series in $z^{-1}$. We define a symplectic form on $\mathcal H$:
$$
\Omega(f,g)=\Res_{z=0}\left(f(-z),g(z)\right)
$$
where $(\mathord{\cdot},\mathord{\cdot})$ is the Poincar\'e paring on $H^*(X_{3,3})$. In this way $\mathcal H$ is polarized as 
$$
\mathcal H^{\GW}=\mathcal H^{\GW}_+ \oplus \mathcal H^{\GW}_-,
$$
with $\mathcal H^{\GW}_+=H^*(X_{3,3})\otimes \mathbb C[z]$ and $\mathcal H^{\GW}_-=z^{-1}H^*(X_{3,3})\otimes \mathbb C[[z^{-1}]]$, and can be regarded as the total cotangent space of $\mathcal H^{\GW}_+$. An element of $\mathcal H^{\GW}$ can be expressed in Darboux coordinates $\{q^\alpha_k,p_{l,\beta}\}$ as 
$$
\sum_{\substack{\alpha\\k\ge0}} q^\alpha_k \phi_\alpha z^k+\sum_{\substack{\beta\\l\ge0}} p_{l,\beta} \phi^\beta (-z)^{-l-1},
$$
where $\{\phi_\alpha\}$ is a bass for $H^*(X_{3,3})$ and $\{\phi^\beta\}$ is its dual bass under Poincar\'e duality. Set
$$
\boldsymbol{q}=\sum_{\substack{\alpha\\k\ge0}} q^\alpha_k \phi_\alpha z^k;
$$
we regard $\mathcal F^0_{\GW}$ as function on $\mathcal H_+$ after the dilaton shift $\boldsymbol{q}=\boldsymbol{t}-z$. In this way, the genus-$0$ Gromov--Witten theory is encoded by a Lagrangian cone 
$$
\mathcal L_{\GW}=\{(\boldsymbol{p},\boldsymbol{q})\colon \boldsymbol{p}=d_{\boldsymbol{q}} F^0_{\GW}\}\subset T^*\mathcal H_+^{\GW}\cong \mathcal H^{\GW}.
$$
At every point $f\in \mathcal L_{\GW}$, the tangent space $T_f\mathcal L_{\GW}$ satisfies the geometric condition \cite{coates2007quantum}
$$
zT_f\mathcal L_{\GW}=\mathcal L_{\GW}\cap T_f\mathcal L_{\GW}.
$$
Therefore $\mathcal L_{\GW}$ is ruled by a family of subspaces 
$$
\{zL\colon L \text{ is a tangent space to } \mathcal L_{\GW}\}.
$$
The $J$-function $J_{\GW}$ is the $\mathcal H^{\GW}$-valued function of $\tau\in H^*(X_{3,3})$ defined by
$$
J_{\GW}(\tau,-z)=-z+\tau+\sum_{n,d,\alpha}\frac{1}{n!}\left\langle\tau,\dots,\tau,\frac{\phi_\alpha}{-z-\varphi_{n+1}}\right\rangle_{0,n+1,d}^{\GW}\phi^\alpha\in -z+\tau+\mathcal H_-^{\GW}.
$$
It can be interpreted as the intersection of $\mathcal L_{\GW}$ with the slice $\{-z+\tau+\mathcal H_-^{\GW}\}$. According to \cite{coates2007quantum}, the partial derivatives of $J_{\GW}(\tau,-z)$ in directions in $H^*(X_{3,3})$ generate the tangent space $T_{J_{\GW}(\tau,-z)}\mathcal L_{\GW}$; also, the cone $\mathcal L_{\GW}$ is ruled by the family of subspaces
$$
\{zT_{J_{\GW}(\tau,-z)}\mathcal L_{\GW} \colon \tau \in H^*(X_{3,3})\}.
$$
In this sense, the $J$-function $J_{\GW}(\tau,-z)$ determines the cone $\mathcal L_{\GW}$. The small $J$-function is defined by restricting $J_{\GW}(\tau,-z)$ to the degree-2 component. Because the degree-2 component of $H^*(X_{3,3})$ is generated by the hyperplane class $p$, the small $J$-function is a function of $t_0^2$; so it is a function of $v$ (see remark \ref{Qe1}). Because $X_{3,3}$ is a Calabi--Yau 3-fold, the virtual dimension of  $\overline{\mathcal M}_{0,n}(X_{3,3},d)$ is $n$; then we can reconstruct $\mathcal L_{\GW}$ from the small $J$-function using the same argument as in \cite{chiodo2010landau}.
\subsubsection*{Restricted theory}
According to the Lefschetz principle, we only consider the theory defined on the classes generated by the hyperplane section $p$. The state space in this theory is
$$
H_{\GW}=\mathbb C \mathds{1} \oplus \mathbb C p \oplus \mathbb C p^2 \oplus \mathbb C p^3 \subset H^*(X_{3,3}).
$$
Note that this subspace coincides with the subspace of compact type according to \cite{clader2017higherwallcrossing}.
The grading on $H_{\GW}$ agrees with the grading on $H^*(X_{3,3})$,\textit{ i.e.}
$$
\Gr(p^n)=2n.
$$
We define another degree $\deg_0$ by setting 
$$
\deg_0(p^n)=2n.
$$
It is called ``bare'' degree in \cite{chiodo2014landau} because it is the degree without age shift. It agrees with the grading since $X_+$ is a smooth variety and the age shift vanishes here.
\subsubsection*{Twisted theory}
Let $\overline{\mathcal M}_{0,n}(\mathbb P^5,d)$ be the moduli spaces of genus-$0$ degree-d n-marked stable maps to $\mathbb P^5$, $\overline{\mathcal C}_{0,n}(\mathbb P^5,d)$ be the universal curve over it, and $\ev$ be the evaluation map as in the following diagram. 

\begin{equation*}
\xymatrix{&\overline{\mathcal C}_{0,n}(\mathbb P^5,d)\ar[r]^-{\ev}\ar[d]^\pi&\mathbb P^5\\&\overline{\mathcal M}_{0,n}(\mathbb P^5,d)}
\end{equation*}
Then $\overline{\mathcal M}_{0,n}(X_{3,3},d)$ is the intersection of the zero locus of two sections of the vector bundle $\pi_*\ev^*\mathcal O_{\mathbb P^5}(3).$ For any choice of $\phi_1,\dots,\phi_n\in H_{\GW} $, we can rewrite the GW invariants as
\begin{equation}\label{twgw}
\langle \tau_{a_1}(\phi_1),\dots,\tau_{a_{n-1}}(\phi_{n-1}),\tau_{a_n}(\phi_n)
\rangle_{0,n,d}^{\GW}=\int_{\overline{\mathcal M}_{0,n}(\mathbb P^5,d)}e\left(\left(\pi_*\ev^*\mathcal O_{\mathbb P^5}(3)\right)^{\oplus 2}\right)\prod_{s=1}^n(\psi_s^{a_s}\ev_s^*\phi_s).
\end{equation}
We can replace the Euler class in (\ref{twgw}) by any multiplicative characteristic class, we define a twisted theory. We can also define their corresponding symplectic vector spaces, Lagrange cones and $J$-functions like above. There are two special cases:
\begin{itemize}
	\item We twist by equivariant Euler class. We denote its corresponding symplectic vector space, Lagrange cone and $J$-function by $\mathcal H^{{\GW},\text{tw}}$,$\mathcal L_{\GW}^\text{tw}$ and $J_{\GW}^\text{tw}$.
	\item We twist by trivial characteristic class, which is identical $\mathds{1}$. This theory is essentially the Gromov--Witten theory of $\mathbb P^5$. We denote its corresponding symplectic vector space, Lagrange cone and $J$-function by $\mathcal H^{{\GW},\text{un}}$,$\mathcal L_{\GW}^\text{un}$ and $J_{\GW}^\text{un}$.
\end{itemize}
We know the Gromov--Witten theory of $\mathbb P^5$ so we know $J_{\GW}^\text{un}$. There is a modification of $J_{\GW}^\text{un}$, which lies on $\mathcal L_{\GW}^\text{tw}$ and determines $\mathcal L_{\GW}^\text{tw}$ (see \cite{coates2007quantum}); we denote it by $I_{\GW}^{\text{eq}}$. So $J_{\GW}^\text{tw}$ and $I_{\GW}^{\text{eq}}$ determines the same cone $\mathcal L_{\GW}^\text{tw}$. We denote the non-equivariant limit of  $J_{\GW}^\text{tw}$ and $I_{\GW}^{\text{eq}}$ by  $J_{\GW}^\text{tw,noneq}$ and $I_{\GW}$; they also determines the same cone. Finally, we have the relation
$$
e\left(\left(\pi_*\ev^*\mathcal O_{\mathbb P^5}(3)\right)^{\oplus 2}\right)J_{\GW}^\text{tw,noneq}(\tau,z)=j_*J_{\GW}(j^*\tau,z),
$$
where $j\colon X_{3,3}\to \mathbb P^5$ is the inclusion. Then $I_{\GW}$ determines $\mathcal L_{\GW}$ in this sense.

The small $I$-function $I_{\GW}$ was computed in \cite{coates2009quantum}. It is given by:
\begin{equation}\label{IGW}
I_{\GW}(v,z)=zv^{\frac{p}{z}}\sum_{n\ge 0}v^{n}\frac{\left(\prod_{0< b \le 3n}(3p+bz)\right)^2}{\left(\prod_{0< b \le n}(p+bz)\right)^6}.
\end{equation}
It is analytic on $|v|<3^{-6}$.

\subsection{Fan--Jarvis--Ruan--Witten theory of $(\C^6,W_1,W_2)$}
The FJRW theory for $(X_-,W_1,W_2)$ is totally parallel to the GW theory of $X_{3,3}$. It was developed in \cite{clader2017landau}. As defined in \S \ref{secgt}, the full state space of FJRW theory is
$$
H^{*-4}_{\text{CR}}(X_-,F_-)=H^{*-4}(X_-,F_-)\oplus H^{*}(\mathbb P(3,3))\oplus H^{*+4}(\mathbb P(3,3)).
$$
In order to determine the subspace of compact type, we write 
$$
H^{*-4}_{\text{CR}}(X_-,X_-\backslash X_-^{\text{cp}})=H^{*-4}(X_-,X_-\backslash X_-^{\text{cp}})\oplus H^{*}(\mathbb P(3,3))\oplus H^{*+4}(\mathbb P(3,3)).
$$
The morphism 
$$
H^{*-4}_{\text{CR}}(X_-,X_-\backslash X_-^{\text{cp}})\to H^{*-4}_{\text{CR}}(X_-,F_-)
$$
are isomorphisms when restricted to the last two direct summands. It is showed in \cite{chiodo2015hybrid} that $H^k_{\text{CR}}(X_-,F_-)=0$ if $k\ne 7$. So it is a zero morphism  when restricted on the first direct summand. Therefore, the subspace of compact type $H_{\RW}$ is 
$$
H^{*}(\mathbb P(3,3))\oplus H^{*+4}(\mathbb P(3,3)).
$$
It coincides with the ``narrow'' part in \cite{clader2017landau}, where the ``narrow'' part is defined to be the part coming from the compact components. Let $H^{(1)}$ and $H^{(2)}$ be the hyperplane classes in the first and second $\mathbb P(3,3)$, then we can write
$$
H_{\RW}=H^{*}(\mathbb P(3,3))\oplus H^{*}(\mathbb P(3,3))=\mathbb C \mathds{1}^{(1)}\oplus \mathbb C H^{(1)} \oplus \mathbb C \mathds{1}^{(2)}\oplus \mathbb C H^{(2)}.
$$
 The grading on $H_{\RW}$ is given by
$$
\Gr(\mathds 1^{(1)})=0,\Gr( H^{(1)})=2,\Gr(\mathds 1^{(2)})=4,\Gr( H^{(2)})=6.
$$
As in the GW theory,  we define the ``bare'' degree $\deg_0$ by ignoring the age shift (see Appendix), then we have
$$
\deg_0(\mathds 1^{(1)})=\deg_0(\mathds 1^{(2)})=-4,\deg_0(H^{(1)})=\deg_0(H^{(2)})=-2.
$$
The genus-$0$ FJRW theory also depends on the $I$-function. It was computed in \cite{clader2017landau} that
\begin{equation}\label{IFJRW}
\begin{split}
I_{\RW}(u,z)=&\sum_{\substack{d\ge 0\\d \not \equiv -1 \mod 3}}\frac{zu^{d+1+\frac{H^{(d+1)}}{z}}}{3^{6[\frac{d}{3}]}}\frac{\displaystyle \prod_{\substack{0<b\leq d\\ b \equiv d+1 \mod 3}}(H^{(d+1)}+bz)^6}{\displaystyle\prod_{\substack{0<b\leq d}}(H^{(d+1)}+bz)^2}\mathds 1^{(d+1)},\\
\end{split}
\end{equation}
where $H^{(h)}=H^{(h \mod 3)}$ if $h\ge 3$. It is analytic on $|u|<3^2$.
\par
Clader showed in \cite{clader2017landau} that $I_{\RW}$ and $I_{GW}$ satisfy the same degree-$4$ differential equation (respect to the variable $u$) after a change of variables $v=u^{-3}$. This equation is the Picard--Fuchs equation corresponding to $X_{3,3}$. By this argument Clader deduced
\begin{thm}[Clader \cite{clader2017landau}]\label{pp}
	There is a $\mathbb C[z,z^{-1}]$-valued degree-preserving linear transformation mapping $I_{\RW}$ to the analytic continuation of $I_{\GW}$.
\end{thm}
Clader only showed the existence of such linear map. In the next section, we will simplify the $I$-functions, and get a family of explicit $\mathbb C$-valued linear maps relating the simplified $I$-functions by a different method. This will allow us to relate these linear maps to equivalences of certain categories in \S \ref{seccom}.

\section{Analytic continuation}\label{secanacon}
	In this section, we introduce the $\mathfrak H$-functions, which are constant linear transform of the $I$-functions. Then we compute the analytic continuation of $\mathfrak H_{\GW}$ and compare it with $\mathfrak H_{\RW}$. In this way we find a linear map 
	$$
	\mathbb U\colon H_{\RW} \to H_{\GW}
	$$
	which identifies $\mathfrak H_{\RW}$ with the analytic continuation $\mathfrak H_{\GW}$.
\subsection{The $\mathfrak H$-functions}
	 We introduce the $\mathfrak H$-functions as in \cite{chiodo2014landau}. 
	 In both GW theory and FJRW theory, the $\mathfrak H$-function is defined by the formula
	 \begin{equation}\label{hdf}
	I=z^{-\frac{\Gr}{2}}\left(\Gamma\cdot(2\pi i)^{\frac{\deg_0}{2}}\mathfrak H\right),
	\end{equation}
	where $I$, $\Gr$ and $\deg_0$ are the $I$-function, grading and ``bare'' degree in the corresponding theory; $\Gamma$ is a chosen class in the corresponding state space.
	
\subsubsection*{Computation of $\mathfrak H_{\GW}$}
	The class $\Gamma_{\GW}$ is chosen to be the Gamma class (see Appendix) of the tangent bundle of $X_{3,3}$. Using the  exact sequence 
	$$
	0\to i^*\mathcal O_{\mathbb P^5}(-3)^{\oplus 2} \to i^*\Omega_{\mathbb P^5}\to \Omega_{X_{3,3}}\to 0
	$$
	and the Euler sequence
	$$
	0 \to \Omega_{\mathbb P^5} \to \mathcal O_{\mathbb P^5}(-1)^{\oplus 6} \to \mathcal O_{\mathbb P^5}\to 0
	$$
	we get
	\begin{equation*}
	\Gamma_{\GW}=\frac{\Gamma(1+p)^6}{\Gamma(1+3p)^2}.
	\end{equation*}
	We rewrite (\ref{IGW}) as
	\begin{equation*}
		\begin{split}
		I_{\GW}(v,z)
			&=\sum_{n \ge 0}zv^{\frac{p}{z}+n}\frac{\Gamma(\frac{p}{z}+1)^6\Gamma(\frac{3p}{z}+3n+1)^2}{\Gamma(\frac{3p}{z}+1)^2\Gamma(\frac{p}{z}+n+1)^6}.\\
		\end{split}
	\end{equation*}
	
	Then
	\begin{equation*}
		\begin{split}
				I_{\GW}(v,z)&=z^{-\frac{\Gr}{2}}\sum_{n\ge 0} zv^{p+n} \frac{\Gamma(p+1)^6}{\Gamma(3p+1)^2}\frac{\Gamma(3p+3n+1)^2}{\Gamma(p+n+1)^6}\\
				&=z^{-\frac{\Gr}{2}}\left(\Gamma_{\GW}\cdot(2\pi i)^{\frac{\deg_0}{2}}\sum_{d\ge 0}zv^{\frac{p}{2\pi i}+n}\frac{\Gamma(3\frac{p}{2\pi i}+3n+1)^2}{\Gamma(\frac{p}{2\pi i}+n+1)^6}\right).
		\end{split}
	\end{equation*}
	Compare with (\ref{hdf}), we have
	\begin{equation}
		\mathfrak H_{\GW}(v,z)=\sum_{n\ge 0}zv^{\frac{p}{2\pi i}+n}\frac{\Gamma(3\frac{p}{2\pi i}+3n+1)^2}{\Gamma(\frac{p}{2\pi i}+n+1)^6}.
	\end{equation}

\subsubsection*{Computation of $\mathfrak H_{\RW}$}
	The Gamma class $\Gamma_{\RW}$ is chosen to be the narrow part of the Gamma class of the tangent bundle of $\mathcal O_{\mathbb P(3,3)}(-1)^{\oplus 6}$, which is
	\begin{equation}
	\Gamma_{\RW}=\Gamma\left(\frac{2}{3}-\frac{H^{(1)}}{3}\right)^6\Gamma(1+H^{(1)})^2\mathds{1}^{(1)}+\Gamma\left(\frac{1}{3}-\frac{H^{(2)}}{3}\right)^6\Gamma(1+H^{(2)})^2\mathds{1}^{(2)}.
	\end{equation}
	We can rewrite (\ref{IFJRW}) as
	\begin{equation}
		\begin{split}
			I_{\RW}(u,z)=&z\sum_{\substack{d\ge 0 \\ d\not\equiv -1 \mod 3}}u^{d+1+\frac{H^{(d+1)}}{z}}z^{-6\langle\frac{d}{3}\rangle}\frac{\Gamma(\frac{H^{(d+1)}}{3z}+\frac{d}{3}+\frac{1}{3})^6\Gamma(\frac{H^{(d+1)}}{z}+1)^2}{\Gamma(\frac{H^{(d+1)}}{3z}+\langle\frac{d}{3}\rangle+\frac{1}{3})^6\Gamma(\frac{H^{(d+1)}}{z}+d+1)^2}\mathds 1^{(d+1)}\\
			=&z\sum_{\substack{d\ge 0 \\ d\equiv 0 \mod 3}}		u^{d+1+\frac{H^{(1)}}{z}}\frac{\Gamma(\frac{H^{(1)}}{3z}+\frac{d}{3}+\frac{1}{3})^6\Gamma(\frac{H^{(1)}}{z}+1)^2}{\Gamma(\frac{H^{(1)}}{3z}+\frac{1}{3})^6\Gamma(\frac{H^{(1)}}{z}+d+1)^2}\mathds 1^{(1)}	\\
			&+z\sum_{\substack{d\ge 0 \\ d\equiv 1 \mod 3}}u^{d+1+\frac{H^{(2)}}{z}}z^{-2}\frac{\Gamma(\frac{H^{(2)}}{3z}+\frac{d}{3}+\frac{2}{3})^6\Gamma(\frac{H^{(2)}}{z}+1)^2}{\Gamma(\frac{H^{(2)}}{3z}+\frac{1}{3})^6\Gamma(\frac{H^{(2)}}{z}+d+1)^2}\mathds 1^{(2)}.
		\end{split}
	\end{equation}
		
	By (\ref{hdf}) we have 
	\begin{equation}\label{ohrw}
		\begin{split}
		(2\pi i)^{-2}\mathfrak H_{\RW}(u,z)=&z\sum_{\substack{d\ge 0 \\ d\equiv 0 \mod 3}}		u^{d+1+H^{(1)}}\frac{\Gamma(1+H^{(1)})^2}{\Gamma(1+H^{(1)})^2}\cdot\frac{\Gamma(\frac{d}{3}+\frac{1}{3}+\frac{H^{(1)}}{3})^6}{\Gamma(\frac{1}{3}+\frac{H^{(1)}}{3})^6\Gamma(\frac{2}{3}-\frac{H^{(1)}}{3})^6\Gamma(d+1+H^{(1)})^2}\mathds 1^{(1)}	\\
		&+z\sum_{\substack{d\ge 0 \\ d\equiv 1 \mod 3}}u^{d+1+H^{(2)}}\frac{\Gamma(1+H^{(2)})^2}{\Gamma(1+H^{(2)})^2}\cdot\frac{\Gamma(\frac{d}{3}+\frac{1}{3}+\frac{H^{(2)}}{3})^6}{\Gamma(\frac{2}{3}+\frac{H^{(2)}}{3})^6\Gamma(\frac{1}{3}-\frac{H^{(2)}}{3})^6\Gamma(d+1+H^{(2)})^2}\mathds 1^{(2)}\\
		=&z\sum_{\substack{d\ge 0 \\ d\equiv 0 \mod 3}}		u^{d+1+H^{(1)}}\frac{\Gamma(\frac{d}{3}+\frac{1}{3}+\frac{H^{(1)}}{3})^6}{\Gamma(\frac{1}{3}+\frac{H^{(1)}}{3})^6\Gamma(\frac{2}{3}-\frac{H^{(1)}}{3})^6\Gamma(d+1+H^{(1)})^2}\mathds 1^{(1)}\\
		&+z\sum_{\substack{d\ge 0 \\ d\equiv 1 \mod 3}}u^{d+1+H^{(2)}}\frac{\Gamma(\frac{d}{3}+\frac{1}{3}+\frac{H^{(2)}}{3})^6}{\Gamma(\frac{2}{3}+\frac{H^{(2)}}{3})^6\Gamma(\frac{1}{3}-\frac{H^{(2)}}{3})^6\Gamma(d+1+H^{(2)})^2}\mathds 1^{(2)}.
		\end{split}
	\end{equation}

	\subsection{Linear maps relating the $\mathfrak H$-functions}
	We can regard $\mathfrak H_{\GW}$ as function of $\log v$ by writing $v$ as $e^{\log v}$. Then $\mathfrak H_{\GW}$ is analytic on $\Re (\log v)<-6\log 3$. In the same way we can regard $\mathfrak H_{\RW}$ as function of $\log u$. Then $\mathfrak H_{\RW}$ is analytic on $\Re (\log v)>-6\log 3$ after a change of variable $\log v=-3\log u$. We can extend $\mathfrak H_{\GW}$ analytically to the right side of the line $\Re(\log v) = -6\log 3$ through the window $w_l$ as in figure \ref{path}, and compare it with $\mathfrak H_{\RW}$. In fact, they are related by the following linear maps:
	\begin{figure}[htbp] 
		\begin{center} 
			\begin{picture}(300,160)  
			\put(150,20){\makebox(0,0){$\circ$}}
			\put(150,60){\makebox(0,0){$\circ$}}  
			\put(150,100){\makebox(0,0){$\circ$}}  
			\put(150,140){\makebox(0,0){$\circ$}} 
			\put(150,0){\dashline[2]{5}(0,0)(0,70)}
			\put(150,102){\dashline[2]{5}(0,0)(0,65)}
			
			
			\put(175,20){\makebox(0,0){\scriptsize $2(l-2)\pi i$}}
			\put(175,60){\makebox(0,0){\scriptsize $2(l-1)\pi i$}}
			\put(164,100){\makebox(0,0){\scriptsize $2l \pi i$}}
			\put(175,140){\makebox(0,0){\scriptsize $2(l+1)\pi i$}}
			\put(147,-10){\makebox(0,0){\scriptsize $\Re(\log v) = -6\log 3$}}
			
			\put(70,80){\makebox(0,0){$\mathfrak H_{\GW}$}}
			\put(230,80){\makebox(0,0){$\mathfrak H_{\RW}$}} 
			\put(150,80){\makebox(0,0){window $w_l$}} 
			
			\put(150,85){\vector(0,1){13}}
			\put(150,75){\vector(0,-1){13}}
			
			\matrixput(10,7)(20,0){7}(0,20){8}{$\cdot$}
			\end{picture} 
		\end{center} 
		\caption{the $(\log v)$-plane.}
		\label{path} 
	\end{figure}
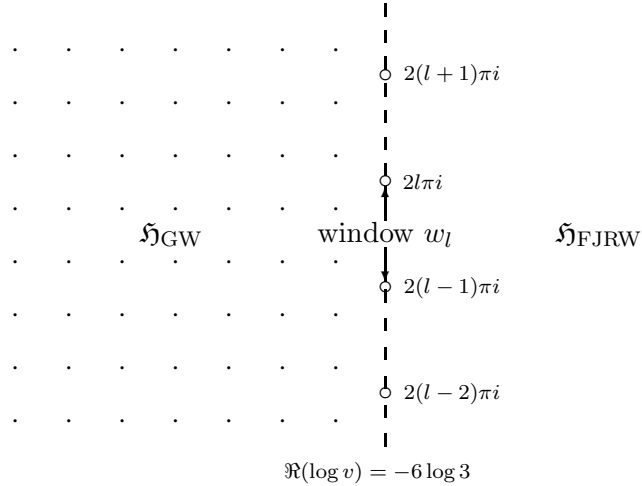

	\begin{dfn}
		For each $l\in \mathbb Z$, the linear map $\mathbb U_l\colon H_{\RW} \to H_{\GW}$ is defined by
		\begin{equation}\label{lnmap}
		\begin{split}
		&\mathds 1^{(1)} \mapsto\frac{l}{9}\frac{\left(\zeta e^{p}\right)^l}{1-\zeta e^{p}}
		+\frac{1}{9}\frac{\left(\zeta e^{p}\right)^{l+1}}{\left(1-\zeta e^{p}\right)^2} \\
		& H^{(1)} \mapsto \frac{1}{3}\frac{\left(\zeta e^{p}\right)^l}{1-\zeta e^{p}}\\
		&\mathds 1^{(2)} \mapsto \frac{l}{9}\frac{\left(\zeta^2e^{p}\right)^l}{1-\zeta^2 e^{p}}
		+\frac{1}{9}\frac{\left(\zeta^2e^{p}\right)^{l+1}}{\left(1-\zeta^2e^{p}\right)^2}\\
		& H^{(2)} \mapsto \frac{1}{3}\frac{\left(\zeta^2 e^{p}\right)^l}{1-\zeta^2 e^{p}}
		\end{split}
		\end{equation}
		where $\zeta=e^{\frac{2\pi i}{3}}$.
	\end{dfn}

	\begin{thm} \label{maincomputation}
		For every $l\in \mathbb Z$, $\mathbb U_l(\mathfrak{H}_{\RW}(u,z))$ coincides with the analytic continuation of $\mathfrak{H}_{\GW}(v,z)$ through the window $w_l$ after the change of variable $\log v=-3\log u$.
	\end{thm}
	\begin{rmk}
		We can write down the explicit linear map in theorem \ref{pp} if we recovery the $I$-functions from the $\mathfrak H$-functions by (\ref{hdf}).
	\end{rmk}
	\begin{rmk}
		Using
		\begin{equation*}
		\begin{split}
		\frac{1}{1-x}&=1+x+x^2+x^3+\dots,\\
		\frac{1}{(1-x)^2}&=1+2x+3x^3+\dots,
		\end{split}
		\end{equation*}
		we can add formal elements $\mathds{1}^{(0)}$ and $H^{(0)}$, and rewrite (\ref{lnmap}) as
		\begin{equation}\label{mirmatrix}
		\begin{split}
		\begin{pmatrix}
		\mathds 1^{(0)} &\mathds 1^{(1)}&\mathds 1^{(2)}
		\end{pmatrix} &\mapsto \frac{1}{9}e^{pl}
		\begin{pmatrix}
		1 & e^p & e^{2p} & e^{3p} \dots 
		\end{pmatrix}
		\begin{pmatrix}
		l&l&l\\
		l+1&(l+1)\zeta&(l+1)\zeta^2\\
		l+2&(l+2)\zeta^2&(l+2)\zeta\\
		l+3&l+3&l+3\\
		l+4&(l+4)\zeta&(l+4)\zeta^2\\
		l+5&(l+5)\zeta^2&(l+5)\zeta\\
		\vdots&\vdots&\vdots
		\end{pmatrix}
		\begin{pmatrix}
		1&&\\
		&\zeta^l&\\
		&&\zeta^{2l}
		\end{pmatrix}\\
		\begin{pmatrix}
		H^{(0)} &H^{(1)}&H^{(2)}
		\end{pmatrix} &\mapsto \frac{1}{3}e^{pl}
		\begin{pmatrix}
		1 & e^p & e^{2p} & e^{3p} \dots 
		\end{pmatrix}
		\begin{pmatrix}
		1&1&1\\
		1&\zeta&\zeta^2\\
		1&\zeta^2&\zeta\\
		1&1&1\\
		1&\zeta&\zeta^2\\
		1&\zeta^2&\zeta\\
		\vdots&\vdots&\vdots
		\end{pmatrix}
		\begin{pmatrix}
		1&&\\
		&\zeta^l&\\
		&&\zeta^{2l}
		\end{pmatrix}.
		\end{split}
		\end{equation}
	\end{rmk}
	\begin{proof}[Proof of Theorem \ref{maincomputation}]
	For $l\in \mathbb Z$, consider the function 
	$$
	F_l(s)=ze^{(\frac{p}{2\pi i}+s)\log v}\cdot\frac{\Gamma(3\frac{p}{2\pi i}+3s+1)^2}{\Gamma(\frac{p}{2\pi i}+s+1)^6}\cdot \frac{\pi}{\sin (\pi s)}\cdot e^{-(2l-1)\pi i s}.
	$$ 
	The poles of $F_l(s)$ are of the form $s=k\in \mathbb Z$ or  $3\frac{p}{2\pi i}+3s+1=-d\in \mathbb Z^{\ge 0},d\equiv 0,1 \mod 3$. They are represented by the black dots in figure \ref{contour}. 
	\begin{figure}[h]
		\begin{picture}(200,150)(-200,-10)
		\qbezier[5](-100,0)(0,0)(100,0)
		\qbezier[5](-100,40)(0,40)(100,40)
		\qbezier[5](-100,80)(0,80)(100,80)
		\qbezier[5](-100,120)(0,120)(100,120)
		\put(120,40){\circle*{3}}
		\put(125,35){\footnotesize{2}}
		\put(60,40){\circle*{3}}
		\put(65,35){\footnotesize{0}}
		\put(90,40){\circle*{3}}
		\put(95,35){\footnotesize{1}}
		\put(30,40){\circle*{3}}
		\put(35,35){\footnotesize{-1}}
		\put(0,40){\circle*{3}}
		\put(5,35){\footnotesize{-2}}
		\put(-30,40){\circle*{3}}
		\put(-25,35){\footnotesize{-3}}
		\put(-60,40){\circle*{3}}
		\put(-55,35){\footnotesize{-4}}
		\put(-90,40){\circle*{3}}
		\put(-85,35){\footnotesize{-5}}
		\put(95,95){\circle*{3}}
		\put(85,95){\circle*{3}}
		\put(65,95){\circle*{3}}
		\put(55,95){\circle*{3}}
		\put(35,95){\circle*{3}}
		\put(25,95){\circle*{3}}
		\put(5,95){\circle*{3}}
		\put(-5,95){\circle*{3}}
		\put(-25,95){\circle*{3}}
		\put(-35,95){\circle*{3}}
		\put(-55,95){\circle*{3}}
		\put(-65,95){\circle*{3}}
		\put(-85,95){\circle*{3}}
		\put(-93,95){\circle*{3}}
		\put(50,90){\vector(0,-1){95}}
		\put(50,90){\line(1,0){50}}
		\put(100,90){\line(0,1){10}}
		\put(100,100){\line(-1,0){50}}
		\put(50,100){\line(0,1){30}}
		\end{picture}
		\caption{the $s$-plane}
		\label{contour}
	\end{figure}
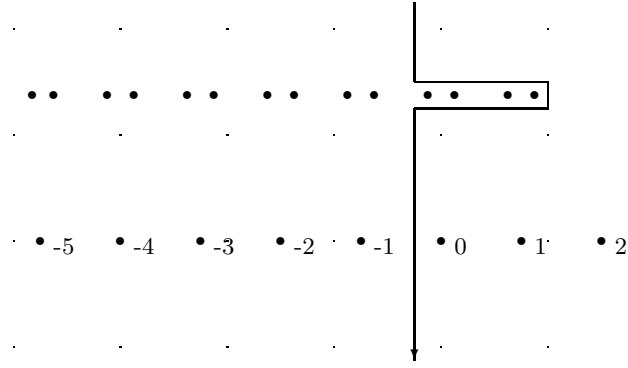
Consider the contour integral $\int_C F_l(s)ds$ along the path of figure \ref{contour}. According to lemma 3.3 in \cite{horja1999hypergeometric}, the integral is absolutely convergent (and defines an analytic function of $v$) if 
$$
\vert \Im(\log v)-(2l-1)\pi\vert<\pi.
$$
Moreover, the integral is equal to the sum of of the residues on the right of the contour for $ \Re (\log v)<-6 \log 3$, and to the opposite of the sum of the residues on the left of the contour for $ \Re (\log v)>-6 \log 3$.

Near the poles $s=k\in \mathbb Z$ we have
	$$
	\frac{\pi}{\sin (\pi s)}\cdot e^{-(2l-1)\pi i s}=\frac{1}{s-k}+O(1),
	$$
	therefore
	\begin{equation}
		\begin{split}
			\mathfrak H_{\GW}(v,z)&=\sum_{n\ge 0}zv^{\frac{p}{2\pi i}+n}\frac{\Gamma(3\frac{p}{2\pi i}+3n+1)^2}{\Gamma(\frac{p}{2\pi i}+n+1)^6}\\
			&=\sum_{n\ge 0} \Res_{s=n} F_l(s)ds\\
			&=\int_C F_l(s)ds
		\end{split}
	\end{equation}
	for $ \Re (\log v)<-6 \log 3$. Then the opposite of the sum of the residues on the left of the contour gives the analytic continuation of $\mathfrak H_{\GW}$ through the windows $w_l$.
	
	In order to compute the residues, we introduce  $\psi$, the logarithmic derivative of the gamma function. It is often called the digamma function, and defined by
	$$
	\psi(z)=\frac{d}{dz}\log(\Gamma(z))=\frac{\Gamma^\prime(z)}{\Gamma(z)}.
	$$
	Near the non-positive integer $-k$ we have the Laurent expansion
	\begin{equation}
		\Gamma(z)=\frac{(-1)^k}{k!}\left(\frac{1}{z+k}+\psi(k+1)\right)+O(z+k).
	\end{equation}
	Thus for a nagetive integer poles $s=n<0$,
	$$
	\Res_{s=n}F_l(s)ds=zv^{\frac{p}{2\pi i}+n}\frac{\Gamma(3\frac{p}{2\pi i}+3n+1)^2}{\Gamma(\frac{p}{2\pi i}+n+1)^6}=0
	$$
	since $$p^4\Big\vert \frac{\Gamma(3\frac{p}{2\pi i}+3n+1)^2}{\Gamma(\frac{p}{2\pi i}+n+1)^6}$$ and $p^4=0$ in $H_{\GW}$.

	The other poles of $F_l(s)ds$ are of the form $3\frac{p}{2\pi i}+3s+1=-d$ for $d\ge 0,d\equiv 0,1 \mod 3$. We calculate the residue at these poles. Near $s=-\frac{p}{2\pi i}-\frac{d}{3}-\frac{1}{3}$, set $s=\Delta s+\frac{p}{2\pi i}+\frac{d}{3}+\frac{1}{3}$, we have 

	\begin{equation}
		\begin{split}
			F_l(s)=&zv^{-\frac{d}{3}-\frac{1}{3}}\left(1+(\log v) \Delta s+O((\Delta s)^2)\right)\\
			&\cdot\frac{1}{(d!)^2} \left(\frac{1}{9}\frac{1}{(\Delta s)^2}+\frac{2\psi(d+1)}{3}\frac{1}{\Delta s} +O(1)\right)\\
			&\cdot \frac{1}{\Gamma(-\frac{d}{3}+\frac{2}{3})^6}\left(1-6\psi(-\frac{d}{3}+\frac{2}{3})\Delta s+O((\Delta s)^2)\right)\\
			&\cdot \frac{\pi}{\sin^2(\frac{p}{2\pi i}+\frac{d}{3}+\frac{1}{3})\pi}\left(-\sin(\frac{p}{2\pi i}+\frac{d}{3}+\frac{1}{3})\pi-\pi \cos(\frac{p}{2\pi i}+\frac{d}{3}+\frac{1}{3})\pi \Delta s+O((\Delta s)^2)\right)\\
			&\cdot e^{(2l-1)\pi i(\frac{p}{2\pi i}+\frac{d}{3}+\frac{1}{3})}\left(1-(2l-1)\pi i \Delta s +O((\Delta s)^2)\right).
		\end{split}
	\end{equation}
	Thus 
	\begin{equation}
		\begin{split}
			\Res_{s=-\frac{p}{2\pi i}-\frac{d}{3}-\frac{1}{3}}F_l=&zv^{-\frac{d}{3}-\frac{1}{3}}\frac{1}{(d!)^2}\frac{1}{\Gamma(-\frac{d}{3}+\frac{2}{3})^6}\frac{\pi}{\sin^2(\frac{p}{2\pi i}+\frac{d}{3}+\frac{1}{3})\pi} e^{(2l-1)\pi i(\frac{p}{2\pi i}+\frac{d}{3}+\frac{1}{3})}\\
			&\cdot\left[\sin(\frac{p}{2\pi i}+\frac{d}{3}+\frac{1}{3})\pi\cdot\left(\frac{2}{3}\psi(-\frac{d}{3}+\frac{2}{3})-\frac{2}{3}\psi(d+1)-\frac{1}{9}\log v+\frac{2l-1}{9}\pi i\right)\right.\\
			&\left.-\cos(\frac{p}{2\pi i}+\frac{d}{3}+\frac{1}{3})\pi\cdot \frac{\pi}{9}\right]\\
			=&zv^{-\frac{d}{3}-\frac{1}{3}}(2\pi i)(\frac{\sqrt 3}{2\pi})^6\frac{\Gamma(\frac{d}{3}+\frac{1}{3})^6}{\Gamma(d+1)^2}\frac{\left(e^{p+(\frac{d}{3}+\frac{1}{3})2\pi i}\right)^l}{e^{p+(\frac{d}{3}+\frac{1}{3})2\pi i}-1}\\
			&\cdot\frac{1}{3}\left(2\psi(\frac{d}{3}+\frac{1}{3})-2\psi(d+1)+2\psi(-\frac{d}{3}+\frac{2}{3})-2\psi(\frac{d}{3}+\frac{1}{3})-\frac{1}{3}\log v\right)\\
			&+zv^{-\frac{d}{3}-\frac{1}{3}}(2\pi i)^2(\frac{\sqrt 3}{2\pi})^6\frac{\Gamma(\frac{d}{3}+\frac{1}{3})^6}{\Gamma(d+1)^2}\frac{\left(e^{p+(\frac{d}{3}+\frac{1}{3})2\pi i}\right)^l}{e^{p+(\frac{d}{3}+\frac{1}{3})2\pi i}-1}\cdot\frac{l}{9}\\
			&-zv^{-\frac{d}{3}-\frac{1}{3}}(2\pi i)^2(\frac{\sqrt 3}{2\pi})^6\frac{\Gamma(\frac{d}{3}+\frac{1}{3})^6}{\Gamma(d+1)^2}\frac{\left(e^{p+(\frac{d}{3}+\frac{1}{3})2\pi i}\right)^{l+1}}{\left(e^{p+(\frac{d}{3}+\frac{1}{3})2\pi i}-1\right)^2}\cdot\frac{1}{9}.
		\end{split}
	\end{equation}
	We used
	$$
	\Gamma(-\frac{d}{3}+\frac{2}{3})^6\Gamma(\frac{d}{3}+\frac{1}{3})^6=\left(\frac{\pi}{\sin(\frac{d}{3}+\frac{1}{3})\pi}\right)^6=\left(\frac{2\pi }{\sqrt{3}}\right)^6
	$$
	and
	$$
	\sin(\frac{p}{2\pi i}+\frac{d}{3}+\frac{1}{3})\pi=\frac{e^{(\frac{p}{2\pi i}+\frac{d}{3}+\frac{1}{3})2\pi i}-1}{2ie^{(\frac{p}{2\pi i}+\frac{d}{3}+\frac{1}{3})\pi i}}.
	$$
 	Then we get the analytic continuation of $\mathfrak H_{\GW}$ through the windows $w_l$, which is
	\begin{equation}\label{afterac}
		\begin{split}
		&\sum_{\substack{d\ge 0\\d \not \equiv -1 \mod 3}}-\Res_{s=-\frac{p}{2\pi i}-\frac{d}{3}-\frac{1}{3}} F_l(s)\\
		=&\sum_{\substack{d\ge 0\\d \equiv 0 \mod 3}}zv^{-\frac{d}{3}-\frac{1}{3}}(2\pi i)(\frac{\sqrt 3}{2\pi})^6\frac{\Gamma(\frac{d}{3}+\frac{1}{3})^6}{\Gamma(d+1)^2}\frac{\left(e^{p+\frac{2\pi i}{3}}\right)^l}{1-e^{p+\frac{2\pi i}{3}}}\\
		&\cdot\frac{1}{3}\left(2\psi(\frac{d}{3}+\frac{1}{3})-2\psi(d+1)+2\psi(\frac{2}{3})-2\psi(\frac{1}{3})-\frac{1}{3}\log v\right)\\
		&+\sum_{\substack{d\ge 0\\d \equiv 0 \mod 3}}zv^{-\frac{d}{3}-\frac{1}{3}}(2\pi i)^2(\frac{\sqrt 3}{2\pi})^6\frac{\Gamma(\frac{d}{3}+\frac{1}{3})^6}{\Gamma(d+1)^2}\\
		&\cdot\left(\frac{\left(e^{p+\frac{2\pi i}{3}}\right)^l}{1-e^{p+\frac{2\pi i}{3}}}\cdot\frac{l}{9}
		+\frac{\left(e^{p+\frac{2\pi i}{3}}\right)^{l+1}}{\left(1-e^{p+\frac{2\pi i}{3}}\right)^2}\cdot\frac{1}{9}\right)\\
		&+\sum_{\substack{d\ge 0\\d \equiv 1 \mod 3}}zv^{-\frac{d}{3}-\frac{1}{3}}(2\pi i)(\frac{\sqrt 3}{2\pi})^6\frac{\Gamma(\frac{d}{3}+\frac{1}{3})^6}{\Gamma(d+1)^2}\frac{\left(e^{p+\frac{4\pi i}{3}}\right)^l}{1-e^{p+\frac{4\pi i}{3}}}\\
		&\cdot\frac{1}{3}\left(2\psi(\frac{d}{3}+\frac{1}{3})-2\psi(d+1)+2\psi(\frac{1}{3})-2\psi(\frac{2}{3})-\frac{1}{3}\log v\right)\\
		&+\sum_{\substack{d\ge 0\\d \equiv 1 \mod 3}}zv^{-\frac{d}{3}-\frac{1}{3}}(2\pi i)^2(\frac{\sqrt 3}{2\pi})^6\frac{\Gamma(\frac{d}{3}+\frac{1}{3})^6}{\Gamma(d+1)^2}\\
		&\cdot\left(\frac{\left(e^{p+\frac{4\pi i}{3}}\right)^l}{1-e^{p+\frac{4\pi i}{3}}}\cdot\frac{l}{9}
		+\frac{\left(e^{p+\frac{4\pi i}{3}}\right)^{l+1}}{\left(1-e^{p+\frac{4\pi i}{3}}\right)^2}\cdot\frac{1}{9}\right),
		\end{split}
	\end{equation}
	where we used
	$$
	\psi(-\frac{d}{3}+\frac{2}{3})-\psi(\frac{d}{3}+\frac{1}{3})=\frac{\pi\cos(\frac{d}{3}-\frac{1}{3})\pi}{\sin(\frac{d}{3}+\frac{1}{3})\pi}=\left\{
	\begin{array}{ll}
	\psi(\frac{2}{3})-\psi(\frac{1}{3})\quad d\equiv 0 \mod 3\\
		\psi(\frac{1}{3})-\psi(\frac{2}{3})\quad d\equiv 1 \mod 3.
	\end{array}
	\right.
	$$
	 On the other hand, we can expand $\mathfrak H_{\RW}$ with respect to $H^{(1)},H^{(2)}$ by differentiating (\ref{ohrw}):
	\begin{equation}\label{hrw}
	\begin{split}
	\mathfrak H_{\RW}&(u,z)=z\sum_{\substack{d\ge 0 \\ d\equiv 0 \mod 3}}u^{d+1}\cdot(2\pi i)^2\cdot(\frac{\sqrt 3}{2\pi})^6\frac{\Gamma(\frac{d}{3}+\frac{1}{3})^6}{\Gamma(d+1)^2}\mathds 1^{(1)}\\
	&+z\sum_{\substack{d\ge 0 \\ d\equiv 0 \mod 3}}u^{d+1}\cdot 2\pi i\cdot(\frac{\sqrt 3}{2\pi})^6\frac{\Gamma(\frac{d}{3}+\frac{1}{3})^6}{\Gamma(d+1)^2}\left(2\psi(\frac{d}{3}+\frac{1}{3})-2\psi(d+1)+2\psi(\frac{2}{3})-2\psi(\frac{1}{3})+\log u\right)H^{(1)}\\
	&+z\sum_{\substack{d\ge 0 \\ d\equiv 0 \mod 3}}u^{d+1}\cdot(2\pi i)^2\cdot(\frac{\sqrt 3}{2\pi})^6\frac{\Gamma(\frac{d}{3}+\frac{1}{3})^6}{\Gamma(d+1)^2}\mathds 1^{(2)}\\
	&+z\sum_{\substack{d\ge 0 \\ d\equiv 0 \mod 3}}u^{d+1}\cdot 2\pi i\cdot(\frac{\sqrt 3}{2\pi})^6\frac{\Gamma(\frac{d}{3}+\frac{1}{3})^6}{\Gamma(d+1)^2}\left(2\psi(\frac{d}{3}+\frac{1}{3})-2\psi(d+1)+2\psi(\frac{1}{3})-2\psi(\frac{2}{3})+\log u\right)H^{(2)}.
	\end{split}
	\end{equation}
	We complete the proof by comparing (\ref{afterac}) with (\ref{hrw}).
	\end{proof}

\section{Orlov functor for complete intersections}\label{secorl}
In this section, we introduce the categories of graded matrix factorizations, and describe a functor between the derived category of graded matrix factorizations and the derived category of $X_{3,3}$.
\subsection{Graded matrix factorizations}\label{secgmf}
\begin{dfn}
	A \textit{Landau--Ginzburg (LG) B-model} is the data of a stack $X$ with a $\mathbb C^*_R$-action, together with a regular function $F$ on $X$, where $-1\in \mathbb C^*_R$ acts trivially on $X$, and $F$ has $\mathbb C^*_R$-weight 2.
\end{dfn}
\begin{exa}
	As in the \S \ref{setup}, we consider a vector space $$V=\C^8=\operatorname{Spec}[x_1,\dots,x_6,p_1,p_2]$$ with a $\mathbb C^*$-action of weights $(1,1,1,1,1,1,-3,-3)$, then there are two different GIT quotients:
	$$
	X_+:=[(\C^6-\{0\})\times \C^2/\C^*]=\mathcal{O}_{\P^5}(-3)^{\oplus 2}
	$$
	and
	$$
	X_-:=[\C^6\times (\C^2-\{0\})/\C^*]=\mathcal{O}_{\P(3,3)}(-1)^{\oplus 6}.
	$$
	We define a $\CR$-action on $V$ to have weights $(0,0,0,0,0,0,2,2)$, then it induces the $\CR$-action on both $X_+$ and $X_-$. Let $W_1$ and $W_2$ be two homogeneous polynomials of degree 3 as in \S \ref{setup}, then the function $W:=p_1W_1+p_2W_2$ on $V$ is invariant under $\C^*$, thus we can regard $W$ as a function on $X_+$ and $X_-$. We get two LG B-models $(X_-,W)$ and $(X_+,W)$ in this way.
\end{exa}
\begin{dfn}
 	A \textit{graded matrix factorization} on a LG B-model $(X,F)$ is a finite rank vector bundle $E$, equivariant with respect to $\mathbb C^*_R$, equipped with an endomorphism $d_E$ of $\mathbb C^*_R$-degree 1 such that $d_E^2=F\cdot \Id_E$.
\end{dfn}
A dg-category $\mathcal{MF}^\CR(X,F)$ is constructed in \cite{segal2011equivalences,shipman2012geometric}, whose objects are graded matrix factorizations over $(X,F)$. We define $\DMF^\CR(X,F)$ to be the homotopy category of  $\mathcal{MF}^\CR(X,F)$, which is a trianglated category.
\begin{rmk}
	In \cite{shipman2012geometric} (or \cite{segal2011equivalences}), graded matrix factorizations are called D-branes (or B-branes). 
\end{rmk}
 Now we describe the structure of triangulated category over $\DMF^\CR(X,F)$. The shift functor on $\DMF^\CR(X,F)$ is given by 
 $$
 (E,d)[1]=(E\otimes \mathcal O[1],-d\otimes \Id)
 $$
 where $\mathcal O[1]$ is the trivial line bundle endowed with a $\CR$-action of weight 1 on fiber direction. Let $f\colon E_1\to E_2$ be a $\CR$-equivariant morphism that intertwines the differentials, then we define the cone 
 $$
 \operatorname{cone}(f\colon E_1\to E_2)=C_f:=\left(E_1[1]\oplus E_2 , \begin{pmatrix}
 d_1[1]&0\\f&d_2
 \end{pmatrix} \right).
 $$ 
 A distinguished triangle is a triangle isomorphic to one of the form
 $$
 E_1\xrightarrow{f} E_2\to C_f\to E_1[1]\to \dots
 $$ 
\begin{exa}[Two important graded matrix factorizations on $(X_\pm,W)$]\label{iptex}
 Before describing the graded matrix factorizations, we need to talk about line bundle over $X_\pm$. Consider a $\C^*$-action over $\C^6\times(\C^2-\{0\})\times \C$ with weights $(1,1,1,1,1,1,-3,-3,k)$, together with a $\CR$-action with weights $(0,0,0,0,0,0,2,2,l)$, then $[\C^6\times(\C^2-\{0\})\times \C/\C^*]$ is a $\CR$-equivariant line bundle over $X_-$. We denote this line bundle by $\mathcal O(k)[l]$. We can define $\mathcal O(k)[l]$ over $X_+$ in the same way. Given a matrix factorization $\mathcal M=(E,d)$ over $(X_\pm,W)$, we denote the matrix factorization $(E\otimes \mathcal O(k)[l],d\otimes \Id)$ by $\mathcal M(k)[l]$.
\par Now we define a distinguished matrix factorization $\mathcal K_-$ over $(X_\pm,W)$ to be the matrix factorization whose underlying $\CR$-equivariant vector bundle is
$$
\bigwedge^\bullet \mathcal O(1)[-1]^{\oplus 6}.
$$
In order to describe its differential, we take $f_1,\dots,f_6$ and $g_1,\dots,g_6$ to be homogeneous polynomials of degree 2 such that $$
W_1=x_1f_{11}+\dots+x_6f_{16}
$$ 
and 
$$
W_2=x_1f_{21}+\dots+x_6f_{26}.
$$
Then $s_x:=(x_1,\dots,x_6)$ is a section of $\mathcal O(1)[0]^{\oplus 6}$, and $s_{pf}:=(p_1f_{11}+p_2f_{21},\dots,p_1f_{16}+p_2f_{26})$ is a cosection of  $\mathcal O(1)[2]^{\oplus 6}$. We define the differential by
$$
d_-(-)=s_x\wedge (-)+ s_{pf} \vee (-). 
$$
Similarly we can define a graded matrix factorization $\mathcal K_+$ over $(X_\pm,W)$ to be the graded D-brane whose underlying $\CR$-equivariant vector bundle is
$$
\bigwedge^\bullet \mathcal O(-3)[1]^{\oplus 2}.
$$
and whose differential is defined by
$$
d_+(-)=s_p\wedge (-)+ s_W \vee (-),
$$
where $s_p:=(p_1,p_2)$ and $s_W:=(W_1,W_2)$.	Note that $\mathcal K_-=0$ in $\DMF^\CR(X_+,W)$ and $\mathcal K_+=0$ in $\DMF^\CR(X_-,W)$.
\end{exa}

Given a graded matrix factorization $(E,d)$, if $E$ can be written as direct sum of sub-bundles and $d$ can be written as sum of the zero extension of morphisms between those sub-bundles, then we can represent $(E,d)$ by a diagram whose vertices are the sub-bundles, and whose arrows are morphism between them. For example, we can represent $\mathcal K_+(q)[m]$ by the diagram
\[
\xymatrix{
	&\mathcal O(q)[m]\ar@<-.3ex>[rr]_-{s_p}&&\mathcal O(q-3)[m+1]^{\oplus 2}\ar@<-.3ex>[ll]_-{s_W} \ar@<-.3ex>[rr]_-{s_p}&&\mathcal O(q-6)[m+2]\ar@<-.3ex>[ll]_-{s_W} .
}
\]
\begin{rmk}\label{swsp}
	Let $A$ be a vector bundle over $X_\pm$, we define the graded matrix factorization $A\otimes \mathcal K_+(q)[m]$ to be $(A(q)[m]\otimes \bigwedge^\bullet \mathcal O(-3)[1]^{\oplus 2},\Id\otimes d_+)$, by an abuse of notation, it can be represented by
	\[
	\xymatrix{
		&A(q)[m]\ar@<-.3ex>[rr]_-{s_p}&&A(q-3)[m+1]^{\oplus 2}\ar@<-.3ex>[ll]_-{s_W} \ar@<-.3ex>[rr]_-{s_p}&&A(q-6)[m+2]\ar@<-.3ex>[ll]_-{s_W} .
	}
	\]
	If we do not require the $\CR$-weight to be 1, then $s_W$ and $s_p$ can be understood as morphisms
		\[
	\xymatrix{
		&A(6)[-3]\ar@<-.3ex>[rr]_-{s_p}&&A(3)[-2]^{\oplus 2}\ar@<-.3ex>[ll]_-{s_W} \ar@<-.3ex>[rr]_-{s_p}&&A\ar@<-.3ex>[ll]_-{s_W}.
	}
	\]
	We use this notation in the next subsection.
	\end{rmk}
\subsection{Orlov functor}\label{secorlovdef}
Orolv \cite{orlov2009derived} constructed a family of equivalences between a category of matrix factorization and derived category of a hypersurface in $\P^4$, we want to generalize it and get a family of equivalences between $\DMF^\CR(X_-,W)$ and $\Db(X_{3,3})$. We can do it by composing a family of functors given by Segal \cite{segal2011equivalences} with a functor given by Shipman \cite{shipman2012geometric}.

\begin{thm}[Segal \cite{segal2011equivalences}]\label{segal}
	There is a family of quasi-equivalences $\Phi_t\colon\mathcal{MF}^\CR(X_-,W)\xrightarrow{\sim}\mathcal{MF}^\CR(X_+,W)$ indexed by $t\in\mathbb Z$.
\end{thm}
When passing to homotopy category, we get a family of equivalences of trianglated category:
 $$\Phi_t\colon\DMF^\CR(X_-,W)\xrightarrow{\sim}\DMF^\CR(X_+,W).$$

\begin{thm}[Shipman \cite{shipman2012geometric}]\label{shipman}
	Let $p:X_+=\mathcal{O}_{\P^5}(-3)^{\oplus 2}\to \P^5$ be the bundle projection, and $i:\mathcal{O}_{X_{3,3}}(-3)^{\oplus 2}\to \mathcal{O}_{\P^5}(-3)^{\oplus 2}$ be the inclusion of total space. Then the functor $i_*\circ p^*\colon \Db(X_{3,3})\to\DMF^\CR(X_+,W)$ is an equivalence of trianglated category.
\end{thm}
Let $\Orl_t$ be the composition of $(i_*\circ p^*)^{-1}\circ\Phi_t$, Then we obtain a family of equivalence 
$$\Orl_t\colon\DMF^\CR(X_-,W) \xrightarrow{\sim}\Db(X_{3,3}).$$ 
\subsubsection*{Description of Shipman's functor}
Shipman's functor $i_*\circ p^*$ can be characterized by the following proposition:
\begin{prop}[\cite{shipman2012geometric}]\label{trick}
	The image of $\mathcal O(k)[l]\in \Db(X_{3,3})$ under the functor $i_*\circ p^*$ in Theorem \ref{shipman} is $\mathcal K_+(k)[l]$. 
\end{prop}
\begin{rmk}\label{rmkkj}
	Using the same trick in the proof of proposition \ref{trick}, we can show $\mathcal K_-(q)[m]$ is the image of $\mathcal O_{\mathbb P_{3,3}}(-q-6)[m-6]$ under the pushforward functor $j_*\colon \Db(\mathbb P_{3,3})\to \DMF^\CR(X_-,W)$.
\end{rmk}
\subsubsection*{Description of Segal's functor}
Segal's functor $\Phi_t$ is constructed in 2 steps:
\begin{enumerate}
	\item 
	Given a graded matrix factorization $(E,d)$ over $(X_-,W)$, we find another graded matrix factorization $(E',d')$ which is isomorphic to $(E,d)$ in $\DMF^\CR(X_-,W)$, and $E'$ is a direct sum of $\mathcal O(k)[l]$ for $t\le k \le t+5$.
	\item 
	Since $\mathcal O(k)[l]$ also stand for line bundles over $X_+$, we take $\Phi_t\left((E,d)\right)$ to be the graded matrix factorization over $(X_+,W)$ with the same shape of direct sum and endomorphism as $(\mathcal E', d')$.
\end{enumerate}
The interval $[t, t+5]$ is called a \textit{window}. In order to apply Segal's functor, we need to find $(E',d')$ in step 1 which lies in the window. We explain the strategy in the next subsection.
\subsection{Strategy to go through the window}
Let $(E,d)$ be a graded matrix factorization, we want to modify it to make it into the window $[t, t+5]$. If $E$ has a direct summand $\mathcal O(k)^{\oplus m}$ which is not in the window, assume $k<t$. Since over $X_-$ we have a resolution
\[
\xymatrix{
	0 \ar[r]&\mathcal O(k+6)^{\oplus m}\ar[r]^{s_p}&\mathcal O(k+3)^{\oplus 2m} \ar[r]^{s_p}&\mathcal O(k)^{\oplus m},
}
\]
we want to replace $\mathcal O(k)^{\oplus m}$ by $\mathcal O(k+6)^{\oplus m}\oplus\mathcal O(k+3)^{\oplus 2m}$. If we can do it repeatedly, we can kill all direct summands outside the window, and finally get a graded matrix factorization in the window. The following lemma and proposition show when and how can we replace a direct summand.

\begin{dfn}\label{defrpl}
 Let $\mathcal M$ be a graded matrix factorization over $(X_\pm,W)$, and $A$ be a direct summand of the underlying line bundle of $\mathcal M$. We say $A$ is \textit{replaceable} in $\mathcal M$ if $\mathcal M$ can be represented by the diagram
 	\begin{equation*}
 \xymatrix{
 	& A   \ar@<-.3ex>[rrr]_{d_{BA}+p_1\delta_{BA}^1+p_2\delta_{BA}^2}   &&&B   \ar@<-.3ex>[lll]_{p_1\delta_{AB}^1+p_2\delta_{AB}^2} \ar@<-.3ex>[rrr]_{d_{CB}+p_1\delta_{CB}^1+p_2\delta_{CB}^2}  & &&C   \ar@<-.3ex>[lll]_{d_{BC}+p_1\delta_{BC}^1+p_2\delta_{BC}^2}\ar@{>}@(ur,dr)^{d_{CC}+p_1\delta_{CC}^1+p_2\delta_{CC}^2},  }
 \end{equation*}
 such that
 \begin{enumerate}
 	\item if we write $A$, $B$ and $C$ as direct sum of $\mathcal O(k)[l]$, then all morphisms $d$ and $\delta$ with some indexes can be represented by matrices with entries in $\mathbb C[x_1,\dots,x_6]$;
 	\item the following equations hold
 	$$
 	\delta_{AB}^2\delta_{BA}^1=\delta_{AB}^1\delta_{BA}^2=0.
 	$$
 \end{enumerate}
\end{dfn}

\begin{lem}\label{iptlm}
	If $A$ is replaceable in $\mathcal M$ as in definition \ref{defrpl}, Then with the notation $s_W$ and $s_p$ in remark \ref{swsp}, the diagram 
	\begin{equation*}
	 \xymatrix{
		A(3)[-2]^{\otimes 2}\ar@<-.3ex>[dd]_{-(d_{BA}+p_1\delta_{BA}^1+p_2\delta_{BA}^2)\circ s_p}\ar@<-.3ex>[rrr]_{-s_W}&&&A(6)[-3]\ar@<-.3ex>[lll]_{-s_p}
		\\\\  B   \ar@<-.3ex>[uu]_{-(\delta_{AB}^1,\delta_{AB}^2)} \ar@<-.3ex>[rrr]_{d_{CB}+p_1\delta_{CB}^1+p_2\delta_{CB}^2}  & &&C   \ar@<-.3ex>[lll]_{d_{BC}+p_1\delta_{BC}^1+p_2\delta_{BC}^2}\ar@{>}@(ur,dr)^{d_{CC}+p_1\delta_{CC}^1+p_2\delta_{CC}^2}\ar[uu]_{\delta_{AB}^1\delta_{BC}^2}  }
\end{equation*}
	represents a graded matrix factorization over $(X_\pm,W)$. We denote the new graded matrix factorization by $\mathcal M \backslash A$.
\end{lem}
\begin{proof}
	It is easy to check the morphisms in the diagram have $\CR$-weight 1. We need to prove the square of sum of them equals to $W\cdot\Id$, \textit {i.e.}
	\begin{enumerate}
		\item For each vertex, the sum of arrows going out composed with their reverse equals $W\cdot \Id$. Because $\mathcal M$ and $A\otimes K_+(6)[-3]$ are graded matrix factorization, it is true at the vertices $A(6)[-3]$ and $C$.
		Note that
		$$
		(p_1\delta_{AB}^1+p_2\delta_{AB}^2)(d_{BA}+p_1\delta_{BA}^1+p_2\delta_{BA}^2)=p_1W_1+p_2W_2,
		$$
		since the sections $p_1,p_2,x_1,\dots,x_6$ are algebraic independent, we deduce
		$$
		(\delta_{AB}^1,\delta_{AB}^2)\circ(d_{BA}+p_1\delta_{BA}^1+p_2\delta_{BA}^2)=s_W.
		$$
		It follows that the property holds at $A(3)[-2]$.
		At $B$, it follows from
		$$
		s_p\circ(\delta_{AB}^1,\delta_{AB}^2)=p_1\delta_{AB}^1+p_2\delta_{AB}^2.
		$$
		\item
		By compositing 2 successive arrows, we get morphisms from one vertex to another. If we fix a pair of different source and target, the sum of those morphisms should be zero. The morphism from $A(6)[-3]$ to $B$ is zero because $s_p\circ s_p=0$ in $A\otimes \mathcal K_+(6)[-3]$. Similarly the morphism froms $A(3)[-2]$ to $C$, from $B$ to $C$, and from $C$ to $B$ are zero. Since $\mathcal M$ is a graded matrix factorization, we have 
		$$
		(d_{BC}+p_1\delta_{BC}^1+p_2\delta_{BC}^2)(d_{CB}+p_1\delta_{CB}^1+p_2\delta_{CB}^2)+(d_{BA}+p_1\delta_{BA}^1+p_2\delta_{BA}^2)(p_1\delta_{AB}^1+p_2\delta_{AB}^2)=p_1W_1+p_2W_2
		$$
		and
		$$
		(p_1\delta_{AB}^1+p_2\delta_{AB}^2)(d_{BC}+p_1\delta_{BC}^1+p_2\delta_{BC}^2)=0,
		$$
		hence
		\[
		\begin{split}
		s_W\circ (\delta_{AB}^1,\delta_{AB}^2)&=-W_2\delta_{AB}^1+W_1\delta_{AB}^2\\
		&=-\delta_{AB}^2d_{BA}\delta_{AB}^1+\delta_{AB}^1d_{BA}\delta_{AB}^2\\
		&=-\delta_{AB}^2(W_1-\delta_{BC}^1d_{CB}-d_{BC}\delta_{CB}^1)+\delta_{AB}^1((W_2-\delta_{BC}^2d_{CB}-d_{BC}\delta_{CB}^2)\\
		&=W_2\delta_{AB}^1-W_1\delta_{AB}^2-2\delta_{AB}^1\delta_{BC}^2d_{CB}
		\end{split}
		\]
		so we get
		\[
		\begin{split}
		&s_W\circ (\delta_{AB}^1,\delta_{AB}^2)+\delta_{AB}^1\delta_{BC}^2(d_{CB}+p_1\delta_{CB}^1+p_2\delta_{CB}^2)\\
		=&(-W_2\delta_{AB}^1+W_1\delta_{AB}^2+\delta_{AB}^1\delta_{BC}^2d_{CB})-p_1\delta_{AB}^2\delta_{BC}^1\delta_{CB}^1+p_2\delta_{AB}^1\delta_{BC}^2\delta_{CB}^2\\
		=&p_1\delta_{AB}^2\delta_{BA}^1\delta_{AB}^1-p_2\delta_{AB}^1\delta_{BA}^2\delta_{AB}^2=0,\\
		\end{split}
		\]
		this proves the sum of morphisms from $B$ to $A(6)[-3]$ is zero.
		We also have
		\[
		\begin{split}
		&(\delta_{AB}^1,\delta_{AB}^2)\circ(d_{BC}+p_1\delta_{BC}^1+p_2\delta_{BC}^2)+s_p\delta_{AB}^1\delta_{BC}^2\\
		=&(p_1\delta_{AB}^1\delta_{BC}^1+p_2\delta_{AB}^1\delta_{BC}^2-p_2\delta_{AB}^1\delta_{BC}^2,p_1\delta_{AB}^2\delta_{BC}^1+p_2\delta_{AB}^2\delta_{BC}^2+p_1\delta_{AB}^1\delta_{BC}^2)\\
		=&(0,0),
		\end{split}
		\]
		this proves the sum morphisms from $C$ to $A(3)[-2]$ is zero. Finally, since
		$$
		(d_{BC}+p_1\delta_{BC}^1+p_2\delta_{BC}^2)(d_{CC}+p_1\delta_{CC}^1+p_2\delta_{CC}^2)=0,
		$$
		we have 
		\[
		\begin{split}
		&\delta_{AB}^1\delta_{BC}^2(d_{CC}+p_1\delta_{CC}^1+p_2\delta_{CC}^2)\\
		=&\delta_{AB}^1\delta_{BC}^2d_{CC}-p_1\delta_{AB}^2\delta_{BC}^1\delta_{CC}^1+p_2\delta_{AB}^1\delta_{BC}^2\delta_{CC}^2\\
		=&-\delta_{AB}^1d_{BC}\delta_{CC}^2=0,
		\end{split}
		\] 
		this proves the morphism from $C$ to $A(6)[-3]$ is zero.
	\end{enumerate}
\end{proof}

\begin{prop}\label{iptpp}
	If $A$ is replaceable in $\mathcal M$, then there exists a morphism of graded matrix factorization 
$$
f\colon \mathcal M \to A\otimes \mathcal K_+ (6)[-2].
$$
Moreover, the cone $C_{f}$ is isomorphic to $\mathcal (M\backslash A)[1]$ in $\DMF^\CR(X_\pm,W)$.
\end{prop}

\begin{proof}
	The morphism between the underlying vector bundles $A\oplus B\oplus C$ and $A\oplus A(3)[-1]^{\oplus 2}\oplus A(6)[-2]$ is given by the matrix $\begin{pmatrix}
	&\Id_A &0&0\\
	&0& (\delta^1_{AB},\delta^2_{AB})&0\\
	&0&0&-\delta^1_{AB} \delta^2_{BC}
	\end{pmatrix}$. We can check it is indeed a morphism of graded matrix factorization by the method used in the proof of lemma \ref{iptlm}. The cone $C_f$ of $f$ is given by
	\begin{equation*}
	\xymatrix{
	&A\ar@<-.3ex>[rrr]_-{ s_W}&&	&A(3)[-1]^{\oplus 2}\ar@<-.3ex>[rrr]_-{s_W} \ar@<-.3ex>[lll]_-{ s_p} &&&A(6)[-2]\ar@<-.3ex>[lll]_-{s_p}\\\\
			&A[1]\ar[uu]_{\Id_A}\ar@<-.3ex>[rrr]_{-(d_{BA}+p_1\delta_{BA}^1+p_2\delta_{BA}^2)}&& &B[1] \ar@<-.3ex>[lll]_{-(p_1\delta_{AB}^1+p_2\delta_{AB}^2)} \ar@<-.3ex>[rrr]_{-(d_{CB}+p_1\delta_{CB}^1+p_2\delta_{CB}^2)}\ar[uu]_{(\delta^1_{AB},\delta^2_{AB})}&&&C[1]\ar@<-.3ex>[lll]_{-(d_{BC}+p_1\delta_{BC}^1+p_2\delta_{BC}^2)}\ar@{>}@(ur,dr)^{-(d_{CC}+p_1\delta_{CC}^1+p_2\delta_{CC}^2)}  \ar[uu]_{-\delta^1_{AB} \delta^2_{BC}}
	}
	\end{equation*}
	We know that $(\mathcal M\backslash A)[1]$ is given by
	\begin{equation*}
	\xymatrix{
	A(3)[-1]^{\otimes 2}\ar@<-.3ex>[dd]_{(d_{BA}+p_1\delta_{BA}^1+p_2\delta_{BA}^2)\circ s_p}\ar@<-.3ex>[rrr]_{s_W}&&&A(6)[-2]\ar@<-.3ex>[lll]_{s_p}
	\\\\  B[1]   \ar@<-.3ex>[uu]_{(\delta_{AB}^1,\delta_{AB}^2)} \ar@<-.3ex>[rrr]_{-(d_{CB}+p_1\delta_{CB}^1+p_2\delta_{CB}^2)}  & &&C[1]   \ar@<-.3ex>[lll]_{-(d_{BC}+p_1\delta_{BC}^1+p_2\delta_{BC}^2)}\ar@{>}@(ur,dr)^{-(d_{CC}+p_1\delta_{CC}^1+p_2\delta_{CC}^2)}\ar[uu]_{-\delta_{AB}^1\delta_{BC}^2}  }
	\end{equation*}
	We write the two underlying vector bundles as direct sums  $$A(6)[-2]\oplus A(3)[-1]^{\oplus 2}\oplus C[1]\oplus B[1]\oplus A\oplus A[1]$$
	and
	$$A(6)[-2]\oplus A(3)[-1]^{\oplus 2}\oplus C[1]\oplus B[1].$$
	In this order, we define a morphism of graded matrix factorization 
	$$F\colon C_f\to (\mathcal M\backslash A)[1]$$
	by the matrix
	$$
	\begin{pmatrix}
	\Id&0&0&0&0&0\\ 0&\Id&0&0&0&0\\ 0&0&\Id&0&0&0\\0&0&0&\Id&d_{BA}+p_1\delta_{BA}^1+p_2\delta_{BA}^2&0
	\end{pmatrix}
	$$ and define
	$$G\colon (\mathcal M\backslash A)[1]\to C_f$$
	by
	$$
	\begin{pmatrix}
	\Id&0&0&0\\0&\Id&0&0\\0&0&\Id&0\\0&0&0&\Id \\0&0&0&0\\0&-s_p &0&0
	\end{pmatrix},
	$$
	we have 
	$$
	F\circ G= \Id_{(\mathcal M\backslash A)[1]}
	$$
	and 
	$$
	G\circ F=\Id_{C_f}+\begin{pmatrix}
	0&0&0&0&0&0\\0&0&0&0&0&0\\0&0&0&0&0&0\\0&0&0&0&d_{BA}+p_1\delta_{BA}^1+p_2\delta_{BA}^2&0\\0&0&0&0&-\Id&0\\0&-s_p&0&0&0&-\Id
	\end{pmatrix}.
	$$
	Define
	$$
	H\colon C_f\to C_f
	$$
	to be
	$$
	\begin{pmatrix}
	0&0&0&0&0&0\\0&0&0&0&0&0\\0&0&0&0&0&0\\0&0&0&0&0&0\\0&0&0&0&0&0\\0&0&0&0&-\Id&0
	\end{pmatrix},
	$$
	then 
	$$
	G\circ F=\Id_{C_f}+H\circ d_{C_f}+d_{C_f}\circ H,
	$$
	which means $G\circ F$ is homotopy to $\Id_{C_f}$, hence we get
	$
	C_f=(\mathcal M\backslash A)[1]
	$ in $\DMF^\CR(X_\pm,W)$.
\end{proof}

\begin{cor} \label{iptrmk}
In $\DMF^\CR(X_-,W)$, we have $\mathcal M=\mathcal M\backslash A$. In this way we replace $A$ by $A(6)[-3]\oplus A(3)[-2]$ as claimed in the beginning of this subsection.
\end{cor}
\begin{proof}
In $\DMF^\CR(X_-,W)$, we have 
	 $$\mathcal K_+=0,$$
	 so 
		 $$
	 \mathcal M=\text{cone}\left(f\colon \mathcal M \to A\otimes \mathcal K_+ (6)[-2] \right)[-1]=\mathcal M\backslash A.
	 $$
\end{proof}

\section{Comparison between Orlov functor and analytic continuation}\label{seccom}
In this section we compute the image of $\mathcal K_-(q)[m]$ under the Orlov functors using the strategy from the previous section. Then we show that the Orlov functors coincide with the linear maps gotten from analytic continuation in the sense that
$$
\ch\left(\Orl_{t-3}(\mathcal K_-(q)[m])\otimes \mathcal O(-3)\right)=\mathbb U_t\left( \ch(\mathcal K_-(q)[m]) \right).
$$

\subsection{Image of $\mathcal K_-(q)[m]$ under Orlov functor}
We compute $\ch\left(\Orl_t(\mathcal K_-(q)[m])\right)$ in this subsection. 
\begin{prop}\label{ppsft}
	For any object $\mathcal F$ in $\DMF^\CR(X_\pm,W)$, and any integer $q,m,t$, we have
	$$
	\Orl_{t}(\mathcal F(q)[m])=\Orl_{t-q}(\mathcal F)(q)[m].
	$$
\end{prop}
\begin{proof}
	If we find a graded matrix factorization $\mathcal E$ such that $\mathcal E=\mathcal F$ in $\DMF^\CR(X_\pm,W)$, and the underlying vector bundle of $\mathcal E$ is a direct sum of $\mathcal O(i)[j]$ for $t\le i \le t+5$, then $\mathcal E(q)[m]=\mathcal F(q)[m]$ in  $\DMF^\CR(X_\pm,W)$, and the underlying vector bundle of $\mathcal E(q)[m]$ is a direct sum of $\mathcal O(i)[j]$ for $t+q\le i \le t+q+5$. By the construction of Segal's functor $\Phi_i$, we have 
	$$
	\Phi_t(\mathcal F)(q)[m]=\Phi_{t+q}(\mathcal F(q)[m]).
	$$
	By compositing with $(i_*\circ p^*)^{-1}$ and using proposition \ref{trick} (note that $\mathcal O(i)[j]$ generate $\Db(X_{3,3})$), we get
	$$
	\Orl_t(\mathcal F)(q)[m]=\Orl_{t+q}(\mathcal F(q)[m]).
	$$
\end{proof}

So we only need to compute $\Orl_t(\mathcal K_-)$. By example \ref{iptex} we can represent $\mathcal K_-$ by
\begin{equation*}
\xymatrix{
\mathcal O \ar@<-.3ex>[r]_-{s_x}
	&\mathcal O(1)[-1]^{\oplus 6}\ar@<-.3ex>[l]_-{s_{pf}}\ar@<-.3ex>[r]_-{s_x}&C\ar@<-.3ex>[l]_-{s_{pf}} \ar@{>}@(ur,dr)^{s_x+s_{pf}}.
}
\end{equation*}
where 
$$
C=\bigwedge^{\ge 2}\mathcal O(1)[-1].
$$

We compute $\Orl_1(\mathcal K_-)$ first. The direct summand $\mathcal O$ of the underlying vector bundle of $\mathcal K_-$ is not in the window $1\le k \le 6$. Fortunately we can decompose $s_{pf}$ as
$$
s_{pf}=p_1s_{f_1}+p_2s_{f_2},
$$
so $\mathcal O$ is replaceable in $\mathcal K_-$. Then by corollary \ref{iptrmk}, in $\DMF^\CR(X_-,W)$, we have a isomorphism between $\mathcal K_-$ and $\mathcal K_-^{(1)}:=\mathcal K_-\backslash \mathcal O$, where $\mathcal K_-^{(1)}$ can be represented by 
\begin{equation*}
\xymatrix{
	\mathcal O(3)[-2]^{\otimes 2}\ar@<-.3ex>[dd]_{-s_x\circ s_p}\ar@<-.3ex>[rrr]_{-s_W}&&&\mathcal O(6)[-3]\ar@<-.3ex>[lll]_{-s_p}
	\\\\  \mathcal O(1)[-1] ^{\oplus 6} \ar@<-.3ex>[uu]_{-(s_{f_1},s_{f_2})} \ar@<-.3ex>[rrr]_{s_x}  & &&C  \ar@<-.3ex>[lll]_{s_{pf}}\ar@{>}@(ur,dr)^{s_x+s_{pf}}\ar[uu]_{0}  }
\end{equation*}
Note that the underlying vector bundle of $\mathcal K_-^{(1)}$ is a direct sum of $\mathcal O(k)[l]$ for $1\le k \le 6$, which are all in the window, thus in $\DMF^\CR(X_+,W)$ we have
$$
\Phi_1(\mathcal K_-)=\mathcal K_-^{(1)}.
$$
By proposition \ref{iptpp} we know in $\DMF^\CR(X_\pm,W)$
$$
\mathcal K_-^{(1)}=\operatorname{cone}( \mathcal K_- \to \mathcal K_+(6)[-2])[-1].
$$
In $\DMF^\CR(X_+,W)$ we have $\mathcal K_-=0$ hence
$$
\mathcal K_-^{(1)}= \mathcal K_+(6)[-3]=i_*\circ p^* (\mathcal O(6)[-3]),
$$
so we have 
$$
\Orl_1(\mathcal K_-)=\mathcal O(6)[-3].
$$	
Next we compute $\Orl_2(\mathcal K_-)$. The window becomes $2\le k \le 7$. The direct summand $\mathcal O(1)[-1]^{\oplus 6}$ of the underlying vector bundle of $\mathcal K_-^{(1)}$ is out the window. Again we can check that $\mathcal O(1)[-1]^{\oplus 6}$ is replaceable in $\mathcal O(1)[-1]^{\oplus 6}$. Let 
$$
\mathcal K_-^{(2)}:= \mathcal K_-^{(1)}\backslash \mathcal O(1)[-1]^{\oplus 6},
$$
then in $\DMF^\CR(X_-,W)$ we have
$$
\mathcal K_-^{(2)}=\mathcal K_-^{(1)}=\mathcal K_-
$$
and all the direct summands of the underlying vector bundle of $\mathcal K_-^{(2)}$ are in the window. Thus in $\DMF^\CR(X_+,W)$ we have 
$$
\Phi_2(\mathcal K_-)=K_-^{(2)}.
$$
By proposition \ref{iptpp},
$$
K_-^{(2)}=\operatorname{cone}\left(\mathcal K_-^{(1)} \to \mathcal K_+(7)[-3]^{\oplus 6}\right)[-1],
$$
so
$$
\Orl_2(\mathcal K_-)=\operatorname{cone}\left( \mathcal O(6)[-4] \to \mathcal O(7)[-4]^{\oplus 6}\right).
$$
If we can repeat above process, then we can compute $\Orl_t(\mathcal K_-)$ for higher $t$. We have
\begin{prop} \label{ahaha}
	There exists a sequence of graded matrix factorizations $\{\mathcal K_-^{(1)},\mathcal K_-^{(2)},\mathcal K_-^{(3)},\dots\}$ in $\DMF^\CR(X_\pm,W)$ such that
	\begin{enumerate}
		\item 
		in $\DMF^\CR(X_-,W)$ we have $\mathcal K_-^{(1)}=\mathcal K_-^{(2)}=\dots=\mathcal K_-$;
		\item
		 the underlying vector space of $\mathcal K_-^{(t)}$ is a direct sum of $\mathcal O(k)[l]$ for $t\le k \le k+5$;
		\item 
		If the direct summand of underlying vector bundle of $\mathcal K_-^{(t)}$ given by all $\mathcal O(t)[l]$ is of the form
		 $$
		 \bigoplus_{i=1}^s \mathcal O(t)[n_i]^{\oplus m_i}, n_1>n_2>\dots>n_s,
		 $$ 
		 then there exists a sequence of graded matrix factorizations $$\{\mathcal K_-^{(t)(0)}=\mathcal K_-^{(t)},\mathcal K_-^{(t)(1)},\mathcal K_-^{(t)(2)},\dots, \mathcal K_-^{(t)(s)}=\mathcal K_-^{(t+1)}\}$$ such that
		 $$
		 \mathcal K_-^{(t)(i+1)}=\operatorname{cone}\left(\mathcal K_-^{(t)(i)}\to \mathcal K_+(t+6)[n_i-2]^{\oplus m_i}\right)[-1]
		 $$
		 in $\DMF^\CR(X_\pm,W)$ for a morphism $\mathcal K_-^{(t)(i)}\to \mathcal K_+(t+6)[n_i-2]^{\oplus m_i}$. The underlying vector bundle of $\mathcal K_-^{(t)(i+1)}$ is obtained by replacing the direct summand $\mathcal O(t)[n_i]^{\oplus m_i}$ in the underlying vector bundle of $\mathcal K_-^{(t)(i)}$ by $\mathcal O(t+6)[n_i-3]^{\oplus m_i}\oplus \mathcal O(t+3)[n_i-2]^{\oplus 2m_i}$.
	\end{enumerate}
\end{prop}
\begin{proof}
	If we can show that $\mathcal O(t)[n_i]^{\oplus m_i}$ is replaceable in $\mathcal K_-^{(t)(i)}$, then we can define 
	$$\mathcal K_-^{(t)(i+1)}:=\mathcal K_-^{(t)(i)}\backslash \mathcal O(t)[n_i]^{\oplus m_i}.$$ We show it is always possible.
	We have written down $\mathcal K_-^{(1)}$ and $\mathcal K_-^{(2)}$ already.  Note that $\mathcal K_-^{(1)}$ and $\mathcal K_-^{(2)}$ satisfy the following property: the differentials of them are the sum of following types of morphisms:
	\begin{enumerate}
		\item 
		the zero extension of the morphism $\mathcal O(k)[l]^{\oplus m} \xrightarrow{f}\mathcal O(k+i)[l-1]^{\oplus n}$, where 
		$$
		i\in \mathbb Z^{\ge 0},
		$$
		and $f$ can be represented by matrix with entries in $\mathbb C[x_1,\dots,x_6]$;
		\item 
		the zero extension of the morphism $\mathcal O(k)[l]^{\oplus m} \xrightarrow{p_1f_1+p_2f_2}\mathcal O(k+i)[l+1]^{\oplus n}$, where 
		$$
		i\in \mathbb Z^{\ge -3},
		$$
		and $f_1,f_2$ can be represented by matrix with entries in $\mathbb C[x_1,\dots,x_6]$.
	\end{enumerate}
	If $\mathcal K_-^{(t)(i)}$ satisfies the property, since all the direct summands of $\mathcal K_-^{(t)(i)}$ is of the form $\mathcal O(s)[j]$ with $s>t$ or $s=t, j\le n_i$, we have all the nonzero arrow with target $\mathcal O(t)[n_i]^{\oplus m_i}$ are of type 2. Moreover, we can not find two successive nonzero type-2 arrows, one of which start from $\mathcal O(t)[n_i]^{\oplus m_i}$ and the other one go back to $\mathcal O(t)[n_i]^{\oplus m_i}$. This means $\mathcal O(t)[n_i]^{\oplus m_i}$ is replaceable in $\mathcal K_-^{(t)(i)}$. By construction in lemma \ref{iptlm}, 
	$$\mathcal K_-^{(t)(i+1)}:=\mathcal K_-^{(t)(i)}\backslash \mathcal O(t)[n_i]^{\oplus m_i}$$ also satisfies above property, thus we can define all $\mathcal K_-^{(t)(i)}$ inductively.
\end{proof}
\begin{rmk}
	We can get another totally similar sequence $\{\mathcal K_-^{(0)},\mathcal K_-^{(-1)},\mathcal K_-^{(-2)},\dots\}$ if we replace $\mathcal O(-t+5)[l]^{\oplus m}$ in the underlying vector bundle of $\mathcal K_-^{(-t)}$ by  $\mathcal O(-t+2)[l+2]^{\oplus 2m}\oplus \mathcal O(-t-1)[l+3]^{\oplus m}$ to get $\mathcal K_-^{(-t-1)}$.
\end{rmk}

From \textit{1} and \textit{2} in proposition \ref{ahaha}, we know 
$$
\Phi_t(\mathcal K_-)=\mathcal K_-^{(t)}
$$
Since $\mathcal K_-^{(t)}$ can be obtained by taking cones of morphism to $\mathcal K_+(k)[n]^{\oplus m}$ repeatedly, after applying the functor $(i_*\circ p^*)^{-1}$, we know $\Orl_t(\mathcal K_-)$ can be obtained by taking cones of morphism to $\mathcal O(k)[n]^{\oplus m}$ repeatedly. In particular, we can compute $\ch(\Orl_t(\mathcal K_-))$.

\begin{thm} For all integers $t,q$ and $m$ we have
	$$\ch(\Orl_t(\mathcal K_-(q)[m]))=(-1)^m\sum_{\substack{t-3\le s \le t+2\\s\equiv q \mod 3}}\frac{s}{3}\sum_{k=0}^{t-s+2}(-1)^{k+1}\binom{6}{k}e^{(k+s+3)p}$$
\end{thm}
\begin{proof}
	We know $\Orl_1(\mathcal K_-)=\mathcal O(6)[-3]$, so $\ch(\Orl_t(\mathcal K_-))=-e^{6p}$. Now we compute $$\ch(\Orl_{t+1}(\mathcal K_-))-\ch(\Orl_t(\mathcal K_-)).$$
	If the direct summand of underlying vector bundle of $\mathcal K_-^{(t)}$ given by all $\mathcal O(t)[l]$ is of the form
	$$
	\bigoplus_{i=1}^s \mathcal O(t)[n_i]^{\oplus m_i},
	$$
	by \textit{3} of proposition \ref{ahaha} there exists a sequence of graded matrix factorizations $$\{\mathcal K_-^{(t)(0)}=\mathcal K_-^{(t)},\mathcal K_-^{(t)(1)},\mathcal K_-^{(t)(2)},\dots, \mathcal K_-^{(t)(s)}=\mathcal K_-^{(t+1)}\}$$ such that
	$$
	\mathcal K_-^{(t)(i+1)}=\operatorname{cone}\left(\mathcal K_-^{(t)(i)}\to K_+(t+6)[n_i-2]^{\oplus m_i}\right)[-1].
	$$
	Therefore, 
	\[
	\begin{split}
	\ch(\Orl_{t+1}(\mathcal K_-))-ch(\Orl_t(\mathcal K_-))&=\ch\left((i_*\circ p^*)^{-1}(\mathcal K_-^{(t)(s)})\right)-	\ch\left((i_*\circ p^*)^{-1}(\mathcal K_-^{(t)(0)})\right)\\
	&=-\sum_{i=1}^s\ch\left(\mathcal O(t+6)[n_i-2]^{\oplus m_i}\right).
	\end{split}
	\]
	When $t\ge 7$, the direct summand $\mathcal O(t)[n_i]^{\oplus m_i}$ in the underlying vector bundle of $\mathcal K_-^{(t)}$ comes from  $\mathcal O(t-3)[n_i+2]^{\oplus a_i}$ in the underlying vector bundle of $\mathcal K_-^{(t-3)}$ and $\mathcal O(t-6)[n_i+3]^{\oplus b_i}$ in the underlying vector bundle of $\mathcal K_-^{(t-6)}$, with $2a_i+b_i=n_i$, therefore when $t\ge 7$ we have
	\[
	\begin{split}
	\ch(\Orl_{t+1}(\mathcal K_-))-\ch(\Orl_t(\mathcal K_-))=&2e^{3p}\left(\ch(\Orl_{t-2}(\mathcal K_-))-\ch(\Orl_{t-3}(\mathcal K_-))\right)\\
	&-e^{6p}\left(\ch(\Orl_{t-5}(\mathcal K_-))-\ch(\Orl_{t-6}(\mathcal K_-))\right).
	\end{split}
	\]
	We can compute directly that
	\[
	\begin{split}
	\ch(\Orl_{2}(\mathcal K_-))-\ch(\Orl_1(\mathcal K_-))&=6e^{7p},\\
	\ch(\Orl_{3}(\mathcal K_-))-\ch(\Orl_2(\mathcal K_-))&=-15e^{8p},\\
	\ch(\Orl_{4}(\mathcal K_-))-\ch(\Orl_3(\mathcal K_-))&=20e^{9p}-2\cdot e^{9p},\\
	\ch(\Orl_{5}(\mathcal K_-))-\ch(\Orl_4(\mathcal K_-))&=-15e^{10p}+2\cdot 6e^{10p},\\
	\ch(\Orl_{6}(\mathcal K_-))-\ch(\Orl_5(\mathcal K_-))&=6e^{11p}-2\cdot 15e^{11p},\\
	\ch(\Orl_{7}(\mathcal K_-))-\ch(\Orl_6(\mathcal K_-))&=-e^{12p}+2\cdot 20e^{12p}-3\cdot e^{12p}.\\
	\end{split}
	\]
	Then we can check that for all $t\ge 1$ we have
	$$\ch(\Orl_{t+1}(\mathcal K_-))-\ch(\Orl_t(\mathcal K_-))=(-1)^{s+1}\lceil\frac{t+1}{3}\rceil\binom{6}{s}+(-1)^{s}\lceil\frac{t+1}{3}-1\rceil\binom{6}{s+3}+(-1)^{s+1}\lceil\frac{t+1}{3}-2\rceil\binom{6}{s+6},$$
	where $s\in \{0,1,2\}$, $s\equiv t \mod 3$, and we set $\binom{6}{k}=0$ if $k>6$. Using the fact that in $H_{\GW}$
	$$
	\sum_{i=0}^6 (-1)^k \binom{6}{k}e^{ip}=(1-e^p)^6=0
	$$
	for dimension reason, we compute
	\[
	\begin{split}
	\ch(\Orl_t(\mathcal K_-))&=\ch(\Orl_1(\mathcal K_-))+\sum_{i=1}^{t-1}\left(\ch(\Orl_{i}(\mathcal K_-))-\ch(\Orl_{i-1}(\mathcal K_-))\right)\\
	&=\sum_{n=1}^{\lceil\frac{t}{3}\rceil}n\sum_{\substack{0\le k \le 6\\3n+k-3\le t}}(-1)^{k+1}\binom{6}{k} e^{(3n+k+3)p}\\
	&=\sum_{\substack{t-3\le s \le t+2\\s\equiv 0 \mod 3}}\frac{s}{3}\sum_{k=0}^{t-s+2}(-1)^{k+1}\binom{6}{k}e^{(s+k+3)p}
	\end{split}
	\]
	Then use proposition \ref{ppsft}, we get
	\[
	\begin{split}
	\ch(\Orl_t(\mathcal K_-(q)[m]))&=(-1)^me^{qp}\ch(\Orl_{t-q}(\mathcal K_-))\\
	&=(-1)^m\sum_{\substack{t-q-3\le s \le t-q+2\\s\equiv 0 \mod 3}}\frac{s}{3}\sum_{k=0}^{t-q-s+2}(-1)^{k+1}\binom{6}{k}e^{(k+s+q+3)p}\\
	&=(-1)^m\sum_{\substack{t-3\le s \le t+2\\s\equiv q \mod 3}}\frac{s-q}{3}\sum_{k=0}^{t-s+2}(-1)^{k+1}\binom{6}{k}e^{(k+s+3)p}.
	\end{split}
	\]
\end{proof}

\subsection{Chern character of $\mathcal K_-(q)[m]$}

The Chern character of $\DMF^\CR(X_-,W)$ (more precisely, of $\mathcal{MF}^\CR(X_-,W)$) takes value in the Hochschild cohomology $HH(\mathcal{MF}^\CR(X_-,W))$. We do not have an isomorphism between 
$HH(\mathcal{MF}^\CR(X_-,W))$ and $H_{\RW}$ in complete intersection case currently. But since all Chern characters satisfy Grothendieck--Riemann--Roth, we can define $H_{\RW}$-value Chern character for objects coming from push-forward as follows. Consider the inclusion 
$$
j\colon \P(3,3) \to X_-=\mathcal O_{\P(3,3)}(-1)^{\oplus 6}
$$
as zero section. The function $W$ is $0$ when restricting to $\P(3,3)$, thus we have the push-forward functor 
$$
j_*\colon \Db(\P(3,3)) \to \DMF^\CR(X_-,W).
$$
Then we use Grothendieck--Riemann--Roth to define
$$
\ch\left(j_*(\mathcal E)\right):=\left.\left(\ch(\mathcal E)\cdot \frac{1}{\operatorname{Td}( O_{\P(3,3)}(-1)^{\oplus 6})}\right)\right\vert_{\text{nar}}.
$$
The Todd class (see Appendix) of the bundle $O_{\P(3,3)}(-1)^{\oplus 6}$ is computed as
\[
\begin{split}
\operatorname{Td}( O_{\P(3,3)}(-1)^{\oplus 6})\vert_{\text{nar}}&=\left(\operatorname{Td}( O_{\P(3,3)}(-1))\right)^6\vert_{\text{nar}}\\
&=\left(\frac{1}{1-\zeta e^{\frac{H^{(1)}}{3}}}\mathds{1}^{(1)}+\frac{1}{1-\zeta^2 e^{\frac{H^{(2)}}{3}}}\mathds{1}^{(2)}\right)^6,
\end{split}
\]
hence
\[
\begin{split}
\left.\left( \frac{1}{\operatorname{Td}( O_{\P(3,3)}(-1)^{\oplus 6})}\right)\right\vert_{\text{nar}}&=\left(\mathds{1}^{(1)}-\zeta e^{\frac{H^{(1)}}{3}}\right)^6+\left(\mathds{1}^{(2)}-\zeta^2 e^{\frac{H^{(2)}}{3}}\right)^6\\
&=\sum_{k=1}^2\left((1-\zeta^k)\mathds 1^{(k)}-\frac{1}{3}\zeta^k H^{(k)}\right)^6.
\end{split}
\]
On the other hand, as mentioned in remark \ref{rmkkj}, $$\mathcal K_-(q)[m]=j_*\left(\mathcal O_{\mathbb P(3,3)}(-q-6)[m-6]\right),$$ 
and
\[
\begin{split}
\ch\left(\mathcal O_{\mathbb P(3,3)}(-q-6)[m-6]\right)\vert_\text{nar}&=(-1)^{m-6}\ch\left(\mathcal O_{\mathbb P(3,3)}(-1)\right)^{q+6}\vert_\text{nar}\\
&=(-1)^m\sum_{k=1}^2\left(\zeta^{-k}\mathds{1}^{(k)}-\frac{1}{3}\zeta^{-k}H^{(k)}\right)^{q+6}.
\end{split}
\]
 So we get
$$
\ch(\mathcal K_-(q)[m])=(-1)^m\sum_{k=1}^2\left(\zeta^{-k}\mathds{1}^{(k)}-\frac{1}{3}\zeta^{-k}H^{(k)}\right)^{q+6}\left((1-\zeta^k)\mathds 1^{(k)}-\frac{1}{3}\zeta^k H^{(k)}\right)^6.
$$
\subsection{Image of $\ch(\mathcal K_-(q)[m])$ under $\mathbb U_l$}

We can expand $\ch(\mathcal K_-(q)[m])$ as
\[
\begin{split}
\ch(\mathcal K_-(q)[m])=&(-1)^m\sum_{k=1}^2 \zeta^{-k(q+6)}(1-\zeta^k)^6\mathds{1}^{(k)}\\
&-2\cdot(-1)^m\sum_{k=1}^2 \zeta^{-k(q+5)}(1-\zeta^k)^5H^{(k)}\\
&-\frac{q+6}{3}(-1)^m\sum_{k=1}^2 \zeta^{-k(q+6)}(1-\zeta^k)^6H^{(k)}.
\end{split}
\]
There are three types of elements in the expansion, we compute their image under $\mathbb U_l$.

\begin{prop}
	For any integer $l$ and $q$, the following 3 equations hold.

		\begin{equation}\label{elem1}
		\begin{split}
		\mathbb U_l \left(\sum_{k=1}^2 \zeta^{-qk}\left(1-\zeta^k\right)^6\mathds{1}^{(k)}\right)
		=&\frac{1}{3}\sum_{\substack{	s\equiv q \mod 3\\l\le s \le l+5}}s\sum_{k=0}^{s-l}(-1)^{l+k+6-s}\begin{pmatrix}
		6\\l+k+6-s
		\end{pmatrix}e^{(k+l)p}\\
		&-2\sum_{\substack{	s\equiv q \mod 3\\l\le s \le l+5}}\sum_{k=0}^{s-l-1}(-1)^{s-l-k}\begin{pmatrix}
		5\\s-l-k-1
		\end{pmatrix}e^{(k+l)p};\\
		\end{split}
		\end{equation}

		\begin{equation}\label{elem2}
		\begin{split}
		&\mathbb U_l\left(\sum_{k=1}^2 \zeta^{-qk}\left(1-\zeta^k\right)^5H^{(k)}\right)
		=\sum_{\substack{	s\equiv q \mod 3\\l\le s \le l+4}}\sum_{k=0}^{s-l}(-1)^{s-l-k}\begin{pmatrix}
		5\\s-l-k
		\end{pmatrix}e^{(k+l)p};\\
		\end{split}
		\end{equation}
 	\begin{equation}\label{elem3}
		\begin{split}
		&\mathbb U_l\left(\sum_{k=1}^2 \zeta^{-qk}\left(1-\zeta^k\right)^6H^{(k)}\right)
		=\sum_{\substack{	s\equiv q \mod 3\\l\le s \le l+5}}\sum_{k=0}^{s-l}(-1)^{s-l-k}\begin{pmatrix}
		6\\s-l-k
		\end{pmatrix}e^{(k+l)p}.\\
		\end{split}
		\end{equation}
\end{prop}
\begin{proof}
	We only prove equation (\ref{elem1}), the other two can be proven in the same way. We add formal element $\mathds 1^{(0)}$ then we have
	\begin{equation}
	\begin{split}
	&\sum_{k=1}^2 \zeta^{-qk}\left(1-\zeta^k\right)^6\mathds{1}^{(k)}\\
	=&\sum_{k=0}^2 \zeta^{-qk}\left(1-\zeta^k\right)^6\mathds{1}^{(k)}\\
	=& \begin{pmatrix}
	\mathds 1^{(0)} &\mathds 1^{(1)}&\mathds 1^{(2)}
	\end{pmatrix}
	\begin{pmatrix}
	1&&\\
	&\zeta^{-q}&\\
	&&\zeta^{-2q}
	\end{pmatrix}
	\begin{pmatrix}
	1&1&1&1&1&1&1\\
	1&\zeta&\zeta^2&1&\zeta&\zeta^2&1\\
	1&\zeta^2&\zeta&1&\zeta^2&\zeta&1
	\end{pmatrix}
	\begin{pmatrix}
	1\\-6\\15\\-20\\15\\-6\\1
	\end{pmatrix}.
	\end{split}
	\end{equation}
	By (\ref{mirmatrix}) we have
	\begin{equation}
	\begin{split}
	&\mathbb U_l \left(\sum_{k=1}^2 \zeta^{-qk}\left(1-\zeta^k\right)^6\mathds{1}^{(k)}\right)\\
	=&\frac{1}{9}e^{pl}
	\begin{pmatrix}
	1 & e^p & e^{2p} & e^{3p} \dots 
	\end{pmatrix}
	\begin{pmatrix}
	l&l&l\\
	l+1&(l+1)\zeta&(l+1)\zeta^2\\
	l+2&(l+2)\zeta^2&(l+2)\zeta\\
	l+3&l+3&l+3\\
	l+4&(l+4)\zeta&(l+4)\zeta^2\\
	l+5&(l+5)\zeta^2&(l+5)\zeta\\
	\vdots&\vdots&\vdots
	\end{pmatrix}
	\begin{pmatrix}
	1&&\\
	&\zeta^l&\\		&&\zeta^{2l}
	\end{pmatrix}\\
	&\cdot\begin{pmatrix}
	1&&\\
	&\zeta^{-q}&\\
	&&\zeta^{-2q}
	\end{pmatrix}
	\begin{pmatrix}
	1&1&1&1&1&1&1\\
	1&\zeta&\zeta^2&1&\zeta&\zeta^2&1\\
	1&\zeta^2&\zeta&1&\zeta^2&\zeta&1
	\end{pmatrix}
	\begin{pmatrix}
	1\\-6\\15\\-20\\15\\-6\\1
	\end{pmatrix}\\
	=&\frac{1}{3}e^{pl}
	\begin{pmatrix}
	1 & e^p & e^{2p} & e^{3p} \dots 
	\end{pmatrix}
	\begin{pmatrix}
	&&l&&&l&\\
	&\reflectbox{$\ddots$}&&&\reflectbox{$\ddots$}&&\\
	m&&&m&&&m\\
	&&m+1&&&m+1&\\
	&m+2&&&m+2&&\\
	m+3&&&m+3&&&m+3\\
	&&\reflectbox{$\ddots$}&&&\reflectbox{$\ddots$}&\\
	&\reflectbox{$\ddots$}&&&\reflectbox{$\ddots$}&&
	\end{pmatrix}
	\begin{pmatrix}
	1\\-6\\15\\-20\\15\\-6\\1
	\end{pmatrix},
	\end{split}
	\end{equation}
	where $m\in\{l,l+1,l+2\}$, and $m\equiv q \mod 3$. Note that the centre matrix consists of slope-1 lines. Start from the third line, the contribution of each line to $\mathbb U_l \left(\sum_{k=1}^2 \zeta^{-qk}\left(1-\zeta^k\right)^6\mathds{1}^{(k)}\right)$ is 
	\begin{equation*}
	\begin{split}
	&\frac{1}{3}e^{pl}\sum_{k=0}^6 (m+3t+k)(-1)^{6-k}\begin{pmatrix}
	6\\6-k
	\end{pmatrix}e^{kp}\\
	=&\frac{1}{3}e^{pl}(m+3t)(1-e^p)^6+\frac{1}{3}e^{pl}\sum_{k=1}^6(-1)^{6-k}\cdot 6 \cdot \begin{pmatrix}
	5\\k-1
	\end{pmatrix}e^{kp}\\
	=&\frac{1}{3}e^{pl}(m+3t)(1-e^p)^6-2e^{(l+1)p}(1-e^p)^5=0.
	\end{split}
	\end{equation*}
	Thus only the first two lines contribute. We have
	\begin{equation}
	\begin{split}
	&\mathbb U_l \left(\sum_{k=1}^2 \zeta^{-qk}\left(1-\zeta^k\right)^6\mathds{1}^{(k)}\right)\\
	=&\frac{1}{3}\sum_{\substack{	s\equiv q \mod 3\\l\le s \le l+5}}\sum_{k=0}^{s-l}(l+k)(-1)^{s-l-k}\begin{pmatrix}
	6\\s-l-k
	\end{pmatrix}e^{(k+l)p}\\
	=&\frac{1}{3}\sum_{\substack{	s\equiv q \mod 3\\l\le s \le l+5}}s\sum_{k=0}^{s-l}(-1)^{l+k+6-s}\begin{pmatrix}
	6\\l+k+6-s
	\end{pmatrix}e^{(k+l)p}\\
	&-\frac{1}{3}\sum_{\substack{	s\equiv q \mod 3\\l\le s \le l+5}}\sum_{k=0}^{s-l}(s-l-k)(-1)^{s-l-k}\begin{pmatrix}
	6\\s-l-k
	\end{pmatrix}e^{(k+l)p}\\
	=&\frac{1}{3}\sum_{\substack{	s\equiv q \mod 3\\l\le s \le l+5}}s\sum_{k=0}^{s-l}(-1)^{l+k+6-s}\begin{pmatrix}
	6\\l+k+6-s
	\end{pmatrix}e^{(k+l)p}\\
	&-2\sum_{\substack{	s\equiv q \mod 3\\l\le s \le l+5}}\sum_{k=0}^{s-l-1}(-1)^{s-l-k}\begin{pmatrix}
	5\\s-l-k-1
	\end{pmatrix}e^{(k+l)p}\\
	\end{split}
	\end{equation}
\end{proof}
\begin{cor}
	The image of $\ch(\mathcal K_-(q)[m]$ under $\mathbb U_l$ is
	\begin{equation}
	\mathbb U_l\left(\ch(\mathcal K_-(q)[m])\right)=(-1)^m\sum_{\substack{	s\equiv q \mod 3\\l\le s \le l+5}}\frac{s-q-6}{3}\sum_{k=0}^{s-l}(-1)^{l+k+6-s}\begin{pmatrix}
	6\\l+k+6-s
	\end{pmatrix}e^{(k+l)p}.
	\end{equation}
\end{cor}
\subsection{Conclusion}

\begin{prop}\label{mp}
	For any integers $q$, $m$ and $t$, we have
	\begin{equation}
	\mathbb U_t\left(\ch(\mathcal K_-(q)[m]) \right)=\ch(\Orl_{t-3}(\mathcal K_-(q)[m]))\cdot e^{-3p}.
	\end{equation}
\end{prop}
\begin{proof}
	Using the fact that 
	$$
	0=(1-e^p)^6=\sum_{k=0}^6(-1)^k\binom{6}{k}e^{kp},
	$$
	we compute
	\begin{equation*}
	\begin{split}
	\mathbb U_t\left(\ch(\mathcal K_-(q)[m]) \right)&=(-1)^m\sum_{\substack{	s\equiv q \mod 3\\t\le s \le t+5}}\frac{s-q-6}{3}\sum_{k=0}^{s-t}(-1)^{t+k+6-s}\begin{pmatrix}
	6\\t+k+6-s
	\end{pmatrix}e^{(k+t)p}\\
	&=(-1)^m\sum_{\substack{	s\equiv q \mod 3\\t\le s \le t+5}}\frac{s-q-6}{3}\sum_{k=s-t-6}^{-1}(-1)^{t+k+5-s}\begin{pmatrix}
	6\\t+k+6-s
	\end{pmatrix}e^{(k+t)p}\\
		&=(-1)^m\sum_{\substack{	s\equiv q \mod 3\\t\le s \le t+5}}\frac{s-q-6}{3}\sum_{k=0}^{t-s+5}(-1)^{k+1}\begin{pmatrix}
	6\\k
	\end{pmatrix}e^{(k+s-6)p}\\
		&=(-1)^m\sum_{\substack{	s\equiv q \mod 3\\t-6\le s \le t-1}}\frac{s-q}{3}\sum_{k=0}^{t-s-1}(-1)^{k+1}\begin{pmatrix}
	6\\k
	\end{pmatrix}e^{(k+s)p}\\
	&=	\ch(\Orl_{t-3}(\mathcal K_-(q)[m]))\cdot e^{-3p}
	\end{split}
	\end{equation*}	
\end{proof}
The functor 
$$
\otimes \mathcal O(-3)\colon \Db(X_{3,3})\to \Db(X_{3,3})
$$	
given by 
$$
\mathcal E \mapsto \mathcal E\otimes \mathcal O(-3)
$$
is an equivalence of triangulated category, so is the composition
$$
\otimes \mathcal O(-3)\circ\Orl_{t-3}\colon \DMF^\CR(X_-,W) \to \Db(X_{3,3}).
$$
Then proposition \ref{mp} can be interpreted as
\begin{thm}\label{maintheorem}
	Let $\mathcal G$ be the subcategory of $\DMF^\CR(X_-,W)$ generated by $\{\mathcal K_-(q)[m]\}_{q,m\in \mathbb Z}$, then for any $t\in \mathbb Z$, we have the following commutative diagram
	\begin{equation*}
	\xymatrix{
	&\mathcal G \ar[rrr]^{\otimes \mathcal O(-3)\circ\Orl_{t-3}}\ar[d]_{\ch}& &&\Db(X_{3,3})\ar[d]^{\ch}\\
	 	 &H_{\RW} \ar[rrr]^{\mathbb U_t}&& &H_{\GW}.
	}
\end{equation*}
\end{thm}

\begin{appendices}
\section{Orbifold cohomology and characteristic classes}
\subsection{Chen--Ruan cohomology}
Let $\mathcal X$ be a smooth Deligne-Mumford stack over $\mathbb C$, and let $I\mathcal X$ be the inertia stack of $\mathcal X$. A point on $I\mathcal X$ is given by a pair $(x,g)$ of a point $x\in \mathcal X$ and $g\in \operatorname{Aut}(x)$. Let $T$ be the index set of components of $I\mathcal X$, then
$$
I\mathcal X=\bigsqcup_{v\in T} \mathcal X_v.
$$
Take a point $(x,g)\in I\mathcal X$ and let
$$
T_x\mathcal X= \bigoplus_{0\le f <1}(T_x\mathcal X)_f
$$
be the eigenvalue decomposition of $T_x\mathcal X$ with respect to the action given by $g$, where $g$ acts on $(T_x\mathcal X)_f$ by $e^{2\pi i f}$. We define 
$$
a_{(x,g)}=\sum_{0\le f <1} f \dim (T_x\mathcal X)_f.
$$
This number is independent of the choice of $(x,g)\in \mathcal X_v$, so we can associate a rational number $a_v$ to each connected component $\mathcal X_v$ of $I\mathcal X$. This is called age shifting number.
\begin{dfn}
	The \textit{Chen--Ruan cohomology group} of $\mathcal X$  is the sum of the singular cohomology of $\mathcal X_v, v\in T$, together with the age shift in gradings:
	$$
	H_{\text{CR}}^k(\mathcal X):=\bigoplus_{v\in T} H^{k-2a_v}(\mathcal X_v,\mathbb C).
	$$
\end{dfn}
\subsection{Characteristic classes}
For an orbifold vector bundle $\tilde E$ on the inertia stack $I\mathcal X$, we have an eigenbundle decomposition of  $\tilde E\vert_{\mathcal X_v }$
$$
\tilde E\vert_{\mathcal X_v }=\bigoplus_{0\le f <1} \tilde E_{v,f}
$$
with respect to the action of the stabilizer of $\mathcal X_v$, where $\tilde E_{v,f}$ is the subbundle with eigenvalue $e^{2\pi i f}$. Let $\operatorname{pr}\colon I\mathcal X\to \mathcal X$ be the projection. For an orbifold vector bundle $E$ on $\mathcal X$, let $\{\delta_{v,f,i}\}_{1\le i\le l_{v,f}}$ be the Chern roots of $(\operatorname{pr}^*E)_{v,f}$, where $l_{v,f}$ is the dimension of $(\operatorname{pr}^*E)_{v,f}$.
\begin{dfn} We define some $H_{\text{CR}}^*(\mathcal X)$-value characteristic classes of an orbifold vector bundle $E$ on $\mathcal X$:
	\begin{itemize}{}
		\item
		The \textit{Chern character} of $E$ is defined by
		$$
		\ch (E):=\bigoplus_{v\in T}\sum_{0\le f <1}e^{2\pi i f}\ch\left((\operatorname{pr}^*E)_{v,f}\right).
		$$
		\item
		The \textit{Todd class} of $V$ is defined by
		$$
		\operatorname{Td}(E):=\bigoplus_{v\in T}\prod_{\substack{0<f<1\\1\le i\le l_{v,f}}}\frac{1}{1-e^{-2\pi if}e^{-\delta_{v,f,i}}}\prod_{1\le i\le l_{v,0}}\frac{\delta_{v,0,i}}{1-e^{-\delta_{v,0,i}}}.
		$$
		\item
		The \textit{Gamma class} of $E$ is defined in \cite{iritani2009integral}, which is 
		$$
		\Gamma(E):=\bigoplus_{v\in T}\prod_{0\le f<1}\prod_{i=1}^{l_{v,f}}\Gamma(1-f+\delta_{v,f,i}),
		$$
		the Gamma function on the right-hand side should be expand in series at $1-f>0$. 
	\end{itemize}
\end{dfn}
\end{appendices}
\printbibliography

\vspace{+16 pt}
\noindent Institut de Math\'ematiques de Jussieu -- Paris Rive Gauche \\
\noindent Sorbonne Universit\'e, UMR 7586 CNRS,\\
\noindent yizhen.zhao@imj-prg.fr 
\end{document}